\documentclass{amsart}
\usepackage{amssymb,latexsym,amsmath,amscd,graphicx,graphics,epic,eepic,bm,color,array,mathrsfs,fullpage}
\usepackage{enumerate}
\usepackage[all,knot,poly]{xy}

\newcommand{\zed}{\mathbb{Z}}

\newcommand{\Q}{\mathbb{Q}}
\newcommand{\fH}{\mathcal{H}}
\newcommand{\fC}{\mathcal{C}}

\newcommand{\Hom}{\mathrm{Hom}}
\newcommand{\im}{\mathrm{Im}}

\newcommand{\ve}{\varepsilon}

\newcommand{\id}{\mathrm{id}}

\newcommand{\mf}{\mathrm{mf}}

\newcommand{\hmf}{\mathrm{hmf}}

\newcommand{\slmf}{\mathfrak{sl}}
\newcommand{\coker}{\mathrm{coker}}

\theoremstyle{plain}
\newtheorem{theorem}{Theorem}[section]
\newtheorem{lemma}[theorem]{Lemma}
\newtheorem{proposition}[theorem]{Proposition}
\newtheorem{corollary}[theorem]{Corollary}

\theoremstyle{definition}
\newtheorem{definition}[theorem]{Definition}

\theoremstyle{remark}
\newtheorem{remark}[theorem]{Remark}

\numberwithin{equation}{section}

\begin{document}

\title[Transverse Khovanov-Rozansky Homologies]{On the Transverse Khovanov-Rozansky Homologies: \\ Graded Module Structure and Stabilization}

\author{Hao Wu}

\thanks{The author was partially supported by NSF grant DMS-1205879.}

\address{Department of Mathematics, The George Washington University, Monroe Hall, Room 240, 2115 G Street, NW, Washington DC 20052, USA. Telephone: 1-202-994-0653, Fax: 1-202-994-6760}

\email{haowu@gwu.edu}



\begin{abstract}
In \cite{Wu-triple-trans}, the author proved that the Khovanov-Rozansky homology $\mathcal{H}_N$ with potential $ax^{N+1}$ is an invariant for transverse links in the standard contact $3$-sphere. In the current paper, we study the $\zed_2\oplus\zed^{\oplus 3}$-graded $\Q[a]$-module structure of $\mathcal{H}_N$, which leads to better understanding of the effect of stabilization on $\mathcal{H}_N$. As an application, we compute $\mathcal{H}_N$ for all transverse unknots.
\end{abstract}

\maketitle

\section{Introduction}\label{sec-intro}

\subsection{The transverse Khovanov-Rozansky homology $\mathcal{H}_N$}

A contact structure $\xi$ on an oriented $3$-manifold $M$ is an oriented tangent plane distribution such that there is a $1$-form $\alpha$ on $M$ satisfying $\xi=\ker\alpha$, $d\alpha|_{\xi}>0$ and $\alpha\wedge d\alpha>0$. Such a $1$-form is called a contact form for $\xi$. The standard contact structure $\xi_{st}$ on $S^3$ is given by the contact form $\alpha_{st} = dz-ydx+xdy=dz+r^2d\theta$. 

We say that an oriented smooth link $L$ in $S^3$ is transverse if $\alpha_{st}|_L>0$. Two transverse links are said to be transverse isotopic if there is an isotopy from one to the other through transverse links.

\begin{theorem}\cite{Ben,OSh,Wr}\label{transversal-markov}
\begin{enumerate}
	\item Every transverse link is transverse isotopic to a counterclockwise transverse closed braid around the $z$-axis.
	\item Any smooth counterclockwise closed braid around the $z$-axis can be smoothly isotoped into a counterclockwise transverse closed braid around the $z$-axis without changing the braid word.
	\item Two counterclockwise transverse closed braids around the $z$-axis are transverse isotopic if and only if the braid word of one of them can be changed into that of the other by a finite sequence of transverse Markov moves. Here, by ``transverse Markov moves", we mean the following braid moves:
	\begin{itemize}
    \item Braid group relations generated by
    \begin{itemize}
	  \item $\sigma_i\sigma_i^{-1}=\sigma_i^{-1}\sigma_i=\emptyset$,
	  \item $\sigma_i\sigma_j=\sigma_j\sigma_i$, when $|i-j|>1$,
	  \item $\sigma_i\sigma_{i+1}\sigma_i=\sigma_{i+1}\sigma_i\sigma_{i+1}$.
    \end{itemize}
    \item Conjugation: $\mu\leftrightsquigarrow\eta^{-1}\mu\eta$,
    where $\mu,~\eta\in \mathbf{B}_m$.\footnote{In this paper, ``$\mathbf{B}_m$" means the braid group on $m$ strands.}
    \item Positive stabilization and destabilization: $\mu~(\in \mathbf{B}_m)\leftrightsquigarrow \mu\sigma_m~(\in \mathbf{B}_{m+1})$.
  \end{itemize}
	In other words, all Markov moves are transverse Markov moves except the negative stabilization and destabilization $\mu~(\in \mathbf{B}_m)\leftrightsquigarrow \mu\sigma_m^{-1}~(\in \mathbf{B}_{m+1})$.
\end{enumerate}
\end{theorem}

Part (1) of Theorem \ref{transversal-markov} was established by Bennequin in \cite{Ben}, part (2) is a simple observation and part (3) was proved by Orevkov, Shevchishin in \cite{OSh} and independently by Wrinkle in \cite{Wr}. Theorem \ref{transversal-markov} means that there is a one-to-one correspondence 
\[
\{\text{Transverse isotopy classes of transverse links}\} \longleftrightarrow \{\text{Closed braids modulo transverse Markov moves}\}.
\] 
So, constructing invariants for transverse links is equivalent to constructing invariants for equivalence classes of closed braids modulo transverse Markov moves. For example, for a closed braid $B$ with writhe $w$ of $m$ strands, its self linking number $sl(B)=w-m$ is invariant under transverse Markov moves. So the self linking number is a transverse link invariant. See \cite{Ben} for the original definition of the self linking number. 

For more about transverse links, see, for example, \cite{Etnyre-contact-notes-1}.

Using the above correspondence, the author introduced in \cite{Wu-triple-trans} a new homological invariant $\fH_N$ for transverse links. $\fH_N$ is a variant of the Khovanov-Rozansky homology defined in \cite{KR1,KR2}. We call $\fH_N$ the $N$th transverse Khovanov-Rozansky homology. The following is the main result of \cite{Wu-triple-trans}.

\begin{theorem}\cite[Theorem 1.2]{Wu-triple-trans}\label{thm-trans-link-homology}
Suppose $N\geq 1$. Let $B$ be a closed braid and $\fC_N(B)$ the chain complex defined in Definition \ref{def-chain-tangle}. Then the homotopy type of $\fC_N(B)$ does not change under transverse Markov moves. Moreover, the homotopy equivalences induced by transverse Markov moves preserve the $\zed_2\oplus \zed^{\oplus 3}$-grading of $\fC_N(B)$, where the $\zed_2$-grading is the $\zed_2$-grading of the underlying matrix factorization and the three $\zed$-gradings are the homological, $a$- and $x$-gradings of $\fC_N(B)$.

Consequently, for the homology $\fH_N(B)=H(H(\fC_N(B),d_{mf}),d_\chi)$ of $\fC_N(B)$ defined in Definition \ref{def-homology-link}, every transverse Markov move on $B$ induces an isomorphism of $\fH_N(B)$ preserving the $\zed_2\oplus \zed^{\oplus 3}$-graded $\Q[a]$-module structure of $\fH_N(B)$.
\end{theorem}

\subsection{Module structure of $\fH_N(B)$} The first part of the current paper is a more careful study of the $\zed_2 \oplus \zed^{\oplus 3}$-graded $\Q[a]$-module structure of $\fH_N(B)$, which refines \cite[Theorem 1.11]{Wu-triple-trans} and leads to Theorem \ref{thm-fH-module} below. 

Before stating Theorem \ref{thm-fH-module}, we introduce the following notations.

\begin{definition}\label{def-component-notations}
Let $B$ be a closed braid. For $(\ve,i,j,k) \in \zed_2 \oplus \zed^{\oplus 3}$, denote by $\fH_N^{\ve,i,j,k}(B)$ the subspace of $\fH_N(B)$ of homogeneous elements of $\zed_2$-degree $\ve$, homological degree $i$, $a$-degree $j$ and $x$-degree $k$. Replacing one of these indices by a ``$\star$" means direct summing over all possible values of this index. For example:
\begin{eqnarray*}
\fH_N^{\ve,i,\star,k}(B) & = & \bigoplus_{j\in\zed} \fH_N^{\ve,i,j,k}(B), \\
\fH_N^{\ve,i,\star,\star} (B) & = & \bigoplus_{(j,k)\in\zed^{\oplus 2}} \fH_N^{\ve,i,j,k}(B).
\end{eqnarray*}

Similarly, for the $\slmf(N)$ Khovanov-Rozansky homology $H_N(B)$ defined in \cite{KR1}\footnote{See Subsection \ref{subsec-KR-sl-N} for our normalization of $H_N(B)$.}, we denote by $H_N^{\ve,i,k}(B)$ the subspace of $H_N(B)$ of homogeneous elements of $\zed_2$-degree $\ve$, homological degree $i$ and $x$-degree $k$. Again, Replacing one of these indices by a ``$\star$" means direct summing over all possible values of this index.
\end{definition}

\begin{theorem}\label{thm-fH-module}
Let $B$ be a closed braid, and $(\ve,i,k) \in\zed_2\oplus \zed^{\oplus2}$. As a $\zed$-graded $\Q[a]$-module, 
\[
\fH_N^{\ve,i,\star,k}(B) \cong (\Q[a]\{sl(B)\}_a)^{\oplus l} \oplus (\Q[a]\{sl(B)+2\}_a)^{\oplus (\dim_\Q H_N^{\ve,i,k}(B) - l)} \oplus (\bigoplus_{q=1}^{n} \Q[a]/(a)\{s_q\}),
\] 
where
	\begin{itemize}
	  \item $\{s\}_a$ means shifting the $a$-grading by $s$,
	  \item $l$ and $n$ are finite non-negative integers determined by $B$ and the triple $(\ve,i,k)$,
	  \item $\{s_1,\dots,s_n\} \subset \zed$ is a sequence determined up to permutation by $B$ and the triple $(\ve,i,k)$,
		\item $sl(B) \leq s_q \leq c_+-c_--1$ and $(N-1)s_q \leq k -2N+ 2c_-$ for $1 \leq q \leq n$, where $c_\pm$ is the number of $\pm$ crossings in $B$.
	\end{itemize}
\end{theorem}

\begin{remark}\label{remark-torsion-s-1}
Note that $sl(B)$ and the number of components of $B$ have the same parity. So, from \cite{KR1}, we know that $H_N^{sl(B)-1,i,k}(B) \cong 0$ and, by Theorem \ref{thm-fH-module}, $\fH_N^{sl(B)-1,i,\star,k}(B)$ is a torsion $\Q[a]$-module.
\end{remark}

\subsection{Stabilization} Applying a negative stabilization to a transverse closed braid $B$, we get a new transverse closed braid $B_-$. In contact geometry, this procedure is called a stabilization of the transverse link. In \cite[Theorem 1.5]{Wu-triple-trans}, the author established that the chain complex $\fC_N(B_-)$ is isomorphic to $cone(\pi_0)\{-2,0\}$, where 
\begin{itemize}
	\item $\pi_0: \fC_N(B) \rightarrow \fC_N(B)/a\fC_N(B)$ is the standard quotient map,
	\item $cone(\pi_0)$ is the mapping cone of $\pi_0$,
	\item $\{j,k\}$ means shifting the $a$-grading by $j$ and the $x$-grading by $k$.
\end{itemize}
Therefore, there is a long exact sequence
{\footnotesize
\[
\cdots \rightarrow \fH_N^{\ve,i-1,\star,\star}(B)\{-2,0\} \xrightarrow{\pi_0} \mathscr{H}_N^{\ve,i-1,\star,\star}(B)\{-2,0\} \rightarrow \fH_N^{\ve,i,\star,\star}(B_-) \rightarrow \fH_N^{\ve,i,\star,\star}(B)\{-2,0\} \xrightarrow{\pi_0} \mathscr{H}_N^{\ve,i,\star,\star}(B)\{-2,0\} \rightarrow \cdots
\]}

\noindent preserving the $a$- and $x$-gradings, where $\mathscr{H}_N(B) := H(H(\fC_N(B)/a\fC_N(B),d_{mf}),d_\chi)$. 

Generally, it is not very easy to compute $\mathscr{H}_N(B)$ even if $\fH_N(B)$ is known. So the above long exact sequence is not very useful when computing the homology of a stabilization of a transverse link. Using Theorem \ref{thm-fH-module}, we will take a closer look at the chain complex $\fC_N(B_-)\cong cone(\pi_0)\{-2,0\}$ and deduce Theorem \ref{thm-stabilization} below.

\begin{theorem}\label{thm-stabilization}
Let $B$ be a closed braid and $B_-$ a stabilization of $B$. Set $s=sl(B)$. Then for any $(i,k)\in \zed^{\oplus 2}$, there are a long exact sequence of $\zed$-graded $\Q[a]$-modules
\begin{equation}\label{eq-exact-seq-stabilization-s-1}
\cdots \rightarrow \fH_N^{s-1,i,\star,k}(B_-) \rightarrow \fH_N^{s,i-1,\star,k+N+1}(B)\{-1\}_a \rightarrow H_N^{s,i-1,k+N+1}(B)\otimes_\Q \Q[a]\{s-1\}_a \rightarrow \fH_N^{s-1,i+1,\star,k}(B_-) \rightarrow \cdots
\end{equation}
and a short exact sequence of $\zed$-graded $\Q[a]$-modules
\begin{equation}\label{eq-exact-seq-stabilization-s}
0 \rightarrow H_N^{s,i,k}(B)\otimes_\Q \Q[a]\{s\}_a \rightarrow \fH_N^{s,i,\star,k}(B_-) \rightarrow \fH_N^{s-1,i-1,\star,k+N+1}(B)\{-1\}_a \rightarrow 0,
\end{equation}
where $H_N(B)$ is the $\slmf(N)$ Khovanov-Rozansky homology of $B$ defined in \cite{KR1}.
\end{theorem}

In \cite{Eliashberg-Fraser}, Eliashberg and Fraser showed that two transverse unknots are transverse isotopic if and only if their self linking numbers are equal. Bennequin's inequality \cite{Ben} implies that the highest self linking number of a transverse unknot is $-1$, which is attained by the $1$-strand transverse closed braid. Denote by $U_0$ the transverse unknot with self linking $-1$ and by $U_m$ the transverse unknot obtained from $U_0$ by $m$ stabilizations. Then every transverse unknot is transverse isotopic to $U_m$ for some $m \geq 0$.

As an application of Theorem \ref{thm-stabilization}, we compute $\fH_N$ for all the transverse unknots. Before stating the result, let us recall that the $\zed$-grading of $\Q[a]$ is given by $\deg_a a =2$. We make $\Q[a]$ a $\zed_2\oplus \zed^{\oplus 3}$-graded $\Q[a]$-module by making the $\zed_2$-, homological and $x$-gradings all $0$ on $\Q[a]$.

\begin{corollary}\label{cor-unknots}
Let $\mathcal{F}$ and $\mathcal{T}$ be the $\zed_2\oplus \zed^{\oplus 3}$-graded $\Q[a]$-modules
\begin{eqnarray*}
\mathcal{F} & := & \bigoplus_{l=0}^{N-1}\Q[a]\left\langle 1\right\rangle \{-1, -N+1+2l\}, \\
\mathcal{T} & := & \bigoplus_{l=0}^{\infty} \Q[a]/(a) \left\langle 1\right\rangle \{-1, N+1+2l\},
\end{eqnarray*}
where ``$\left\langle \ve \right\rangle$" means shifting the $\zed_2$-grading by $\ve$ and ``$\{j,k\}$" means shifting the $a$-grading by $j$ and the $x$-gradings by $k$. Then,
\begin{eqnarray*}
\fH_N(U_0) & \cong & \mathcal{F} \oplus \mathcal{T}, \\
\fH_N(U_1) & \cong & \mathcal{F} \oplus \mathcal{T} \left\langle 1 \right\rangle \{-1, -N-1\}\|1\|,
\end{eqnarray*}
and, for $m \geq 2$,
\[
\fH_N(U_m) \cong \mathcal{F}\{-2(m-1),0\} \oplus \mathcal{T} \left\langle m \right\rangle \{-m, -m(N+1)\}\|m\| \oplus \bigoplus_{l=1}^{m-1} \mathcal{F}/a\mathcal{F} \left\langle l \right\rangle \{-2m+l, -l(N+1)\}\|l+1\|,
\]
where ``$\|l\|$" means shifting the homological grading by $l$.
\end{corollary}

\subsection{Organization of this paper} In Section 2, we review the definition of $\fH_N$. Then we study the $\Q[a]$-module structure of $\fH_N$ and prove Theorem \ref{thm-fH-module} in Section \ref{sec-fH-module}. Finally, we prove Theorem \ref{thm-stabilization} and Corollary \ref{cor-unknots} in Section \ref{sec-stabilization}.

This paper is self-contained for the most part. Of course, some prior knowledge of the Khovanov-Rozansky homology, especially of \cite{KR1,Wu-triple-trans}, will be helpful.

\section{Definition of $\fH_N$}\label{sec-def}

In this section, we quickly review the definition of the transverse Khovanov-Rozansky homology $\fH_N$ in \cite{Wu-triple-trans}, which is every similar to the definition of the Khovanov-Rozansky homology in \cite{KR1,KR2}.

\subsection{$\zed_2\oplus\zed^{\oplus 2}$-graded matrix factorizations over $\Q[a,x_1,\dots,x_k]$}

\begin{definition}\label{def-bigrading-ring}
We define a $\zed^{\oplus2}$-grading on $R=\Q[a,x_1,\dots,x_k]$ by letting $\deg a = (2,0)$ and $\deg x_i =(0,2)$ for $i=1,\dots, k$. We call the first component of this $\zed^{\oplus2}$-grading the $a$-grading and denote its degree function by $\deg_a$. We call the second component of this $\zed^{\oplus2}$-grading the $x$-grading and denote its degree function by $\deg_x$. An element of $R$ is said to be homogeneous if it is homogeneous with respect to both the $a$-grading and the $x$-grading.

A $\zed^{\oplus2}$-graded $R$-module $M$ is a $R$-module $M$ equipped with a $\zed^{\oplus2}$-grading such that, for any homogeneous element\footnote{An element of $M$ is said to be homogeneous if it is homogeneous with respect to both $\zed$-gradings.} $m$ of $M$, $\deg (am) = \deg m + (2,0)$ and $\deg (x_i m) = \deg m +(0,2)$ for $i=1,\dots, k$. Again, we call the first component of this $\zed^{\oplus2}$-grading of $M$ the $a$-grading and denote its degree function by $\deg_a$. We call the second component of this $\zed^{\oplus2}$-grading of $M$ the $x$-grading and denote its degree function by $\deg_x$. 

We say that the $\zed^{\oplus2}$-grading on $M$ is bounded below if both the $a$-grading and the $x$-grading are bounded below.

For a $\zed^{\oplus2}$-graded $R$-module $M$, we denote by $M\{j,k\}$ the $\zed^{\oplus2}$-graded $R$-module obtained by shifting the $\zed^{\oplus2}$-grading of $M$ by $(j,k)$. That is, for any homogeneous element $m$ of $M$, $\deg_{M\{j,k\}} m = \deg_M m + (j,k)$.
\end{definition}

\begin{definition}\label{def-mf}
Let $w$ be a homogeneous element with bidegree $(2,2N+2)$ of $R=\Q[a,x_1,\dots,x_k]$. A $\zed_2\oplus\zed^{\oplus2}$-graded matrix factorization $M$ of $w$ over $R$ is a collection of two $\zed^{\oplus2}$-graded free $R$-modules $M_0$, $M_1$ and two homogeneous $R$-module maps $d_0:M_0\rightarrow M_1$, $d_1:M_1\rightarrow M_0$ of bidegree $(1,N+1)$, called differential maps, such that 
\[
d_1 \circ d_0=w\cdot\id_{M_0}, \hspace{1cm}  d_0 \circ d_1=w\cdot\id_{M_1}.
\]
The $\zed_2$-grading of $M$ takes value $\ve$ on $M_\ve$. The $a$- and $x$-gradings of $M$ are the $a$- and $x$-gradings of the underlying $\zed^{\oplus 2}$-graded $R$-module $M_0 \oplus M_1$.

We usually write $M$ as $M_0 \xrightarrow{d_0} M_1 \xrightarrow{d_1} M_0$.
\end{definition}

Following \cite{KR1}, we denote by $M\left\langle 1\right\rangle$ the matrix factorization $M_1 \xrightarrow{d_1} M_0 \xrightarrow{d_0} M_1$ and write $M\left\langle j\right\rangle = M \underbrace{\left\langle 1\right\rangle\cdots\left\langle 1\right\rangle}_{j \text{ times }}$.

For any $\zed_2\oplus\zed^{\oplus2}$-graded matrix factorization $M$ of $w$ over $R$ and $j,k \in \zed$, $M\{j,k\}$ is naturally a $\zed_2\oplus\zed^{\oplus2}$-graded matrix factorization of $w$ over $R$.

For any two $\zed_2\oplus\zed^{\oplus2}$-graded matrix factorizations $M$ and $M'$ of $w$ over $R$, $M\oplus M'$ is naturally a $\zed_2\oplus\zed^{\oplus2}$-graded matrix factorization of $w$ over $R$.

Let $w$ and $w'$ be two homogeneous elements of $R$ with bidegree $(2,2N+2)$. For $\zed_2\oplus\zed^{\oplus2}$-graded matrix factorizations $M$ of $w$ and $M'$ of $w'$ over $R$, the tensor product $M\otimes_R M'$ is the $\zed_2\oplus\zed^{\oplus2}$-graded matrix factorization of $w+w'$ over $R$ such that:
\begin{itemize}
	\item $(M\otimes M')_0 = (M_0\otimes M'_0)\oplus (M_1\otimes M'_1)$, $(M\otimes M')_1 = (M_1\otimes M'_0)\oplus (M_1\otimes M'_0)$;
	\item The differential is given by the signed Leibniz rule. That is, $d(m\otimes m')=(dm)\otimes m' + (-1)^\ve m \otimes (dm')$ for $m\in M_\ve$ and $m'\in M'$.
\end{itemize}

\begin{definition}\label{def-morph-mf}
Let $w$ be a homogeneous element of $R$ with bidegree $(2,2N+2)$, and $M$, $M'$ any two $\zed_2\oplus\zed^{\oplus2}$-graded matrix factorizations of $w$ over $R$. 
\begin{enumerate}
	\item A morphism of $\zed_2\oplus\zed^{\oplus2}$-graded matrix factorizations from $M$ to $M'$ is a homogeneous $R$-module homomorphism $f:M\rightarrow M'$ preserving the $\zed_2\oplus\zed^{\oplus2}$-grading satisfying $d_{M'}f=fd_{M}$. We denote by $\Hom_\mf(M,M')$ the $\Q$-space of all morphisms of $\zed_2\oplus\zed^{\oplus2}$-graded matrix factorizations from $M$ to $M'$.
	\item An isomorphism of $\zed_2\oplus\zed^{\oplus2}$-graded matrix factorizations from $M$ to $M'$ is a morphism of $\zed_2\oplus\zed^{\oplus2}$-graded matrix factorizations that is also an isomorphism of the underlying $R$-modules. We say that $M$ and $M'$ are isomorphic, or $M\cong M'$, if there is an isomorphism from $M$ to $M'$.
	\item Two morphisms $f$ and $g$ of $\zed_2\oplus\zed^{\oplus2}$-graded matrix factorizations from $M$ to $M'$ are called homotopic if there is an $R$-module homomorphism $h:M\rightarrow M'$ shifting the $\zed_2$-grading by $1$ such that $f-g = d_{M'}h+hd_M$. In this case, we write $f \simeq g$. We denote by $\Hom_\hmf(M,M')$ the $\Q$-space of all homotopy classes of morphisms of $\zed_2\oplus\zed^{\oplus2}$-graded matrix factorizations from $M$ to $M'$. That is, $\Hom_\hmf(M,M') = \Hom_\mf(M,M') / \simeq$.
	\item $M$ and $M'$ are called homotopic, or $M\simeq M'$, if there are morphisms $f:M\rightarrow M'$ and $g:M'\rightarrow M$ such that $g\circ f \simeq \id_M$ and $f\circ g \simeq \id_{M'}$. $f$ and $g$ are called homotopy equivalences between $M$ and $M'$.
  \item We say that $M$ is homotopically finite if it is homotopic to a finitely generated graded matrix factorization of $w$ over $R$. 
\end{enumerate}

We define categories $\mf^{\mathrm{all}}_{R,w}$, $\mf_{R,w}$, $\hmf^{\mathrm{all}}_{R,w}$ and $\hmf_{R,w}$ by the following table. 

\begin{center}
\small{
\begin{tabular}{|c|l|c|}
\hline
Category & Objects & Morphisms \\
\hline
$\mf^{\mathrm{all}}_{R,w}$ & all $\zed_2\oplus\zed^{\oplus2}$-graded matrix factorizations of $w$ over $R$ with the   & $\Hom_{\mf}$ \\
 & $\zed^{\oplus2}$-grading bounded below &  \\
\hline
$\mf_{R,w}$ & all homotopically finite $\zed_2\oplus\zed^{\oplus2}$-graded matrix factorizations of $w$    & $\Hom_{\mf}$ \\
 & over $R$ with the $\zed^{\oplus2}$-grading bounded below &  \\
\hline
$\hmf^{\mathrm{all}}_{R,w}$ & all $\zed_2\oplus\zed^{\oplus2}$-graded matrix factorizations of $w$ over $R$ with the   & $\Hom_{\hmf}$ \\
 & $\zed^{\oplus2}$-grading bounded below &  \\
\hline
$\hmf_{R,w}$ & all homotopically finite $\zed_2\oplus\zed^{\oplus2}$-graded matrix factorizations of $w$    & $\Hom_{\hmf}$ \\
 & over $R$ with the $\zed^{\oplus2}$-grading bounded below &  \\
\hline
\end{tabular}
}
\end{center}
\end{definition}

\begin{definition}\label{def-koszul-mf}
If $a_0,a_1\in R$ are homogeneous elements with $\deg a_0 +\deg a_1=(2,2N+2)$, then denote by $(a_0,a_1)_R$ the $\zed_2\oplus\zed^{\oplus2}$-graded matrix factorization $R \xrightarrow{a_0} R\{1-\deg_a a_0,~N+1-\deg_x{a_0}\} \xrightarrow{a_1} R$ of $a_0a_1$ over $R$. More generally, if $a_{1,0},a_{1,1},\dots,a_{l,0},a_{l,1}\in R$ are homogeneous with $\deg a_{j,0} +\deg a_{j,1}=(2,2N+2)$, then denote by 
\[
\left(%
\begin{array}{cc}
  a_{1,0}, & a_{1,1} \\
  a_{2,0}, & a_{2,1} \\
  \dots & \dots \\
  a_{l,0}, & a_{l,1}
\end{array}%
\right)_R
\]
the tenser product $(a_{1,0},a_{1,1})_R \otimes_R (a_{2,0},a_{2,1})_R \otimes_R \cdots \otimes_R (a_{l,0},a_{l,1})_R$, which is a $\zed_2\oplus\zed^{\oplus2}$-graded matrix factorization of $\sum_{j=1}^l a_{j,0} a_{j,1}$ over $R$, and is call the Koszul matrix factorization associated to the above matrix. We drop``$R$" from the notation when it is clear from the context. 

Note that the above Koszul matrix factorization is finitely generated over $R$.
\end{definition}

The following proposition from \cite{KR1} is useful in computing the homology of some MOY graphs.

\begin{proposition}\cite[Proposition 10]{KR1}\label{prop-contraction-weak}
Let $I$ be an ideal of $R$ generated by homogeneous elements. Assume $w$, $a_0$ and $a_1$ are homogeneous elements of $R$ such that $\deg w=\deg a_0 +\deg a_1 = (2,2N+2)$ and $w+a_0a_1 \in I$. Then $w \in I+(a_0)$ and $w \in I+(a_1)$. 

Let $M$ be a $\zed_2\oplus\zed^{\oplus2}$-graded matrix factorization of $w$ over $R$, and $\widetilde{M}=M \otimes_R (a_0,a_1)_R$. Then ${\widetilde{M}/I\widetilde{M}}$, ${M/(I+(a_0))M}$ and ${M/(I+(a_1))M}$ are all $\zed_2 \oplus\zed^{\oplus2}$-graded chain complexes of $R$-modules.
\begin{enumerate}
	\item If $a_0$ is not a zero-divisor in $R/I$, then there is an $R$-linear quasi-isomorphism $f:{\widetilde{M}/I\widetilde{M}} \rightarrow {(M/(I+(a_0))M)\left\langle 1\right\rangle \{1-\deg_a a_0, N+1-\deg_x a_0 \}}$ that preserves the $\zed_2\oplus\zed^{\oplus 2}$-grading.
	\item If $a_1$ is not a zero-divisor in $R/I$, then there is an $R$-linear quasi-isomorphism $g:{\widetilde{M}/I\widetilde{M}} \rightarrow {M/(I+(a_1))M}$ that preserves the $\zed_2\oplus\zed^{\oplus 2}$-grading.
\end{enumerate}
\end{proposition}

\subsection{The matrix factorization associated to a MOY graph}

\begin{definition}\label{def-MOY}
A MOY graph $\Gamma$ is an oriented graph embedded in the plane satisfying:
\begin{enumerate}
	\item Every edge of $\Gamma$ is colored by $1$ or $2$.
	\item Every vertex of $\Gamma$ is $1$-, $2$- or $3$-valent.
	\item Every $1$-valent vertex of $\Gamma$ is either the initial point of a $1$-colored edge or the terminal point of a $1$-colored edge. We call $1$-valent vertices of $\Gamma$ endpoints of $\Gamma$.
	\item Every $2$-valent vertex of $\Gamma$ is the initial point of a $1$-colored edge and the terminal point of a $1$-colored edge.
	\item Every $3$-valent vertex of $\Gamma$ is 
  \begin{itemize}
		\item either the initial point of two $1$-colored edges and the terminal point of a $2$-colored edge,
		\item or the terminal point of two $1$-colored edges and the initial point of a $2$-colored edge.
	\end{itemize}
\end{enumerate}
\end{definition}

In particular, Definition \ref{def-MOY} means that every $2$-colored edge of $\Gamma$ has a neighborhood that looks like the local configuration in Figure \ref{fig-wide-edge}.

\begin{figure}[ht]

\[
\setlength{\unitlength}{1pt}
\begin{picture}(60,30)(-30,0)

\put(-10,15){\vector(1,0){20}}

\put(-25,0){\vector(1,1){15}}

\put(-25,30){\vector(1,-1){15}}

\put(10,15){\vector(1,-1){15}}

\put(10,15){\vector(1,1){15}}

\put(-2,17){\tiny{$2$}}

\put(-17,0){\tiny{$1$}}

\put(-17,25){\tiny{$1$}}

\put(15,0){\tiny{$1$}}

\put(15,25){\tiny{$1$}}

\end{picture}
\]

\caption{}\label{fig-wide-edge}

\end{figure}

\begin{definition}\label{def-MOY-marking}
Let $\Gamma$ be a MOY graph. A marking of $\Gamma$ consists of:
\begin{enumerate}
	\item A finite collection of of marked points on $\Gamma$ such that
  \begin{itemize}
		\item all endpoints are marked,
		\item none of the $2$- or $3$-valent vertices are marked,
		\item every $1$-colored edge contains a marked point\footnote{We consider the initial and terminal points of an edge part of that edge.},
		\item none of the $2$-colored edges contain marked points.
	\end{itemize}
	\item An assignment that assigns to each marked point a single variable such that no two marked points are assigned the same variable.
\end{enumerate}
\end{definition}

Now suppose $\Gamma$ is a MOY graph with a marking. Let $x_1,\dots, x_m$ be all the variables assigned to marked points on $\Gamma$ and $x_{i_1},\dots,x_{i_n}$ all the variables assigned to $1$-valent vertices of $\Gamma$. We define $R$ to be the $\zed^{\oplus 2}$-graded ring $R=\Q[a,x_1,\dots,x_m]$ with the $\zed^{\oplus 2}$-grading given by $\deg a = (2,0)$ and $\deg x_i = (0,2)$. Denote by $R_\partial$ the $\zed^{\oplus 2}$-graded sub-ring $R_\partial=\Q[a,x_{i_1},\dots,x_{i_n}]$ of $R$. we call $R_\partial$ the boundary ring of the marked MOY graph $\Gamma$.

\begin{figure}[ht]

\[
\xymatrix{
\setlength{\unitlength}{1pt}
\begin{picture}(60,45)(-30,-15)

\put(-20,15){\vector(1,0){40}}

\put(-2,17){\tiny{$1$}}

\put(-30,15){\small{$x_i$}}

\put(23,15){\small{$x_k$}}

\put(-4,-15){$\Gamma_{i;k}$}

\end{picture} && \setlength{\unitlength}{1pt}
\begin{picture}(70,45)(-35,-15)

\put(-10,15){\vector(1,0){20}}

\put(-25,0){\vector(1,1){15}}

\put(-25,30){\vector(1,-1){15}}

\put(10,15){\vector(1,-1){15}}

\put(10,15){\vector(1,1){15}}

\put(-2,17){\tiny{$2$}}

\put(-17,0){\tiny{$1$}}

\put(-17,25){\tiny{$1$}}

\put(15,0){\tiny{$1$}}

\put(15,25){\tiny{$1$}}

\put(-35,0){\small{$x_j$}}

\put(-35,25){\small{$x_i$}}

\put(28,0){\small{$x_l$}}

\put(28,25){\small{$x_k$}}

\put(-4,-15){$\Gamma_{i,j;k,l}$}

\end{picture}
}
\]

\caption{}\label{fig-MOY-pieces}

\end{figure}

Next, cut $\Gamma$ at all of its marked points. This breaks $\Gamma$ into simple marked MOY graphs $\Gamma_1,\cdots,\Gamma_p$, each of which is of one of the two types in Figure \ref{fig-MOY-pieces}. Note that each $\Gamma_q$ is marked only at its endpoints. Denote by $R_q$ the $\zed^{\oplus 2}$-graded polynomial ring over $\Q$ generated by $a$ and the variables marking $\Gamma_q$. 

\begin{itemize}
	\item If $\Gamma_q = \Gamma_{i;k}$ in Figure \ref{fig-MOY-pieces}, then $R_q = \Q[a,x_i,x_k]$ and 
	\begin{equation}\label{eq-def-mf-arc}
	\fC_N(\Gamma_q) = (a\cdot\frac{x_k^{N+1}-x_i^{N+1}}{x_k-x_i},x_k-x_i)_{R_q}.
	\end{equation}
	\item If $\Gamma_q = \Gamma_{i,j;k,l}$ in Figure \ref{fig-MOY-pieces}, then $R_q = \Q[a,x_i,x_j,x_k,x_l]$ and
	\begin{equation}\label{eq-def-mf-wide-edge}
	\fC_N(\Gamma_q) = \left(%
  \begin{array}{cc}
  a\cdot\frac{g(x_k+x_l,x_kx_l)-g(x_i+x_j,x_kx_l)}{x_k+x_l-x_i-x_j}, & x_k+x_l-x_i-x_j \\
  a\cdot\frac{g(x_i+x_j,x_kx_l)-g(x_i+x_j,x_ix_j)}{x_kx_l-x_ix_j}, & x_kx_l-x_ix_j
  \end{array}%
  \right)_{R_q}\{0,-1\},
	\end{equation}
	where $g$ is the unique $2$-variable polynomial satisfying $g(x+y,xy)=x^{N+1}+y^{N+1}$.
\end{itemize}

\begin{definition}\label{def-mf-MOY}
\[
\fC_N(\Gamma) = \bigotimes_{q=1}^p (\fC_N(\Gamma_q) \otimes_{R_q} R),
\]
where the big tensor product ``$\bigotimes_{q=1}^p$" is taken over the ring $R=\Q[a,x_1,\dots,x_m]$.

Note that $\fC_N(\Gamma)$ is a $\zed_2\oplus\zed^{\oplus2}$-graded matrix factorization of $w=\sum_{k=1}^n \pm a x_{i_k}^{N+1}$, where the sign is positive if $\Gamma$ points outward at the corresponding endpoint and negative if $\Gamma$ points inward at the corresponding endpoint. 

We view $\fC_N(\Gamma)$ as an object of the category $\hmf^{\mathrm{all}}_{R_\partial,w}$.
\end{definition}

\begin{definition}\label{def-homology-MOY}
A MOY graph is called closed if it has no endpoints. If $\Gamma$ is a closed MOY graph, then $\fC_N(\Gamma)$ is a $\zed_2 \oplus \zed^{\oplus 2}$-graded matrix factorization of $0$. So it is a homologically $\zed_2$-graded chain complex of $\zed^{\oplus 2}$-graded $\Q[a]$-modules with a homogeneous differential map. We denote by $\fH_N(\Gamma)$ the homology of this chain complex. Note that $\fH_N(\Gamma)$ is a $\zed_2 \oplus \zed^{\oplus 2}$-graded $\Q[a]$-module by inheriting the gradings of $\fC_N(\Gamma)$.
\end{definition}

The following two lemmas are slight generalizations of the corresponding results in \cite{KR1,KR2}.

\begin{lemma}\label{lemma-MOY-decomps}\cite[Corollary 5.6, Lemma 3.11 and Proposition 7.1]{Wu-triple-trans}
As matrix factorizations over the respective boundary rings, we have: \vspace{1pc}
\begin{equation}\label{eq-MOY-I}
\fC_N\left(\setlength{\unitlength}{1pt}
\begin{picture}(61,20)(-31,20)

\put(0,10){\vector(0,1){20}}

\put(-30,10){\line(0,1){20}}

\qbezier(0,30)(0,40)(-20,40)

\qbezier(-30,30)(-30,40)(-20,40)

\qbezier(0,10)(0,0)(-20,0)

\qbezier(-30,10)(-30,0)(-20,0)

\put(20,0){\vector(0,1){40}}

\put(3,20){\tiny{$1$}}

\put(15,20){\tiny{$1$}}

\end{picture}\right) \simeq \fC_N\left(\setlength{\unitlength}{1pt}
\begin{picture}(50,20)(-25,20)

\put(-10,5){\line(2,1){10}}

\put(20,0){\vector(-2,1){20}}

\put(0,30){\vector(-2,1){10}}

\put(0,30){\vector(2,1){20}}

\put(0,10){\vector(0,1){20}}

\qbezier(-10,35)(-20,40)(-20,30)

\qbezier(-10,5)(-20,0)(-20,10)

\put(-20,10){\line(0,1){20}}

\put(15,7){\tiny{$1$}}

\put(-17,30){\tiny{$1$}}

\put(15,33){\tiny{$1$}}

\put(3,20){\tiny{${2}$}}

\end{picture}\right)\{0,1\} \oplus \fC_N\left(\setlength{\unitlength}{1pt}
\begin{picture}(10,20)(15,20)

\put(20,0){\vector(0,1){40}}

\put(15,20){\tiny{$1$}}

\end{picture}\right)\left\langle 1\right\rangle\{-1,1-N\},
\end{equation}

\begin{equation}\label{eq-MOY-II}
\fC_N\left(\setlength{\unitlength}{1pt}
\begin{picture}(32,35)(-77,30)

\put(-60,5){\vector(0,1){10}}

\qbezier(-60,15)(-70,15)(-70,25)

\put(-70,25){\vector(0,1){10}}

\qbezier(-70,35)(-70,45)(-60,45)

\qbezier(-60,15)(-50,15)(-50,25)

\put(-50,25){\vector(0,1){10}}

\qbezier(-50,35)(-50,45)(-60,45)

\put(-60,45){\vector(0,1){10}}

\put(-65,48){\tiny{$2$}}

\put(-65,7){\tiny{$2$}}

\put(-75,30){\tiny{$1$}}

\put(-49,30){\tiny{$1$}}

\put(-70,-5){\vector(1,1){10}}

\put(-50,-5){\vector(-1,1){10}}

\put(-60,55){\vector(1,1){10}}

\put(-60,55){\vector(-1,1){10}}

\put(-75,-5){\tiny{$1$}}

\put(-75,60){\tiny{$1$}}

\put(-49,-5){\tiny{$1$}}

\put(-49,60){\tiny{$1$}}

\end{picture}\right) \simeq \fC_N\left(\setlength{\unitlength}{1pt}
\begin{picture}(32,35)(-77,30)

\put(-60,5){\vector(0,1){50}}

\put(-65,30){\tiny{$2$}}

\put(-70,-5){\vector(1,1){10}}

\put(-50,-5){\vector(-1,1){10}}

\put(-60,55){\vector(1,1){10}}

\put(-60,55){\vector(-1,1){10}}

\put(-75,-5){\tiny{$1$}}

\put(-75,60){\tiny{$1$}}

\put(-49,-5){\tiny{$1$}}

\put(-49,60){\tiny{$1$}}

\end{picture}\right)\{0,-1\} \oplus \fC_N\left(\setlength{\unitlength}{1pt}
\begin{picture}(32,35)(-77,30)

\put(-60,5){\vector(0,1){50}}

\put(-65,30){\tiny{$2$}}

\put(-70,-5){\vector(1,1){10}}

\put(-50,-5){\vector(-1,1){10}}

\put(-60,55){\vector(1,1){10}}

\put(-60,55){\vector(-1,1){10}}

\put(-75,-5){\tiny{$1$}}

\put(-75,60){\tiny{$1$}}

\put(-49,-5){\tiny{$1$}}

\put(-49,60){\tiny{$1$}}

\end{picture}\right)\{0,1\},
\end{equation}

\begin{equation}\label{eq-MOY-III}
\fC_N\left(\setlength{\unitlength}{1pt}
\begin{picture}(70,40)(-45,30)

\put(-40,-10){\vector(1,1){20}}

\put(0,-10){\vector(-1,1){20}}

\put(-20,10){\vector(0,1){10}}

\put(-20,20){\vector(0,1){20}}

\put(-20,40){\vector(0,1){10}}

\put(-20,50){\vector(1,1){20}}

\put(-20,50){\vector(-1,1){20}}

\put(20,-10){\vector(0,1){30}}

\put(20,20){\vector(0,1){20}}

\put(20,40){\vector(0,1){30}}

\put(-20,20){\vector(1,0){40}}

\put(20,40){\vector(-1,0){40}}

\put(-18,43){\tiny{$2$}}

\put(-18,30){\tiny{$1$}}

\put(-18,12){\tiny{$2$}}

\put(-1,33){\tiny{$1$}}

\put(-1,22){\tiny{$1$}}

\put(15,50){\tiny{$1$}}

\put(15,30){\tiny{$2$}}

\put(15,-5){\tiny{$1$}}

\put(0,-5){\tiny{$1$}}

\put(-40,-5){\tiny{$1$}}

\put(0,63){\tiny{$1$}}

\put(-40,63){\tiny{$1$}}

\end{picture}\right) \oplus \fC_N\left(\setlength{\unitlength}{1pt}
\begin{picture}(70,40)(-25,30)

\put(40,-10){\vector(-1,1){20}}

\put(0,-10){\vector(1,1){20}}

\put(-20,-10){\vector(0,1){80}}

\put(20,50){\vector(1,1){20}}

\put(20,50){\vector(-1,1){20}}

\put(20,10){\vector(0,1){40}}

\put(-18,30){\tiny{$1$}}

\put(15,30){\tiny{$2$}}

\put(0,-5){\tiny{$1$}}

\put(40,-5){\tiny{$1$}}

\put(0,63){\tiny{$1$}}

\put(40,63){\tiny{$1$}}

\end{picture}\right) \simeq \fC_N\left(\setlength{\unitlength}{1pt}
\begin{picture}(70,40)(-25,30)

\put(40,-10){\vector(-1,1){20}}

\put(0,-10){\vector(1,1){20}}

\put(-20,-10){\vector(0,1){30}}

\put(-20,20){\vector(0,1){20}}

\put(-20,40){\vector(0,1){30}}

\put(20,50){\vector(1,1){20}}

\put(20,50){\vector(-1,1){20}}

\put(20,10){\vector(0,1){10}}

\put(20,20){\vector(0,1){20}}

\put(20,40){\vector(0,1){10}}

\put(20,20){\vector(-1,0){40}}

\put(-20,40){\vector(1,0){40}}

\put(15,43){\tiny{$2$}}

\put(-18,30){\tiny{$2$}}

\put(15,12){\tiny{$2$}}

\put(-1,33){\tiny{$1$}}

\put(-1,22){\tiny{$1$}}

\put(-18,50){\tiny{$1$}}

\put(15,30){\tiny{$1$}}

\put(-18,-5){\tiny{$1$}}

\put(0,-5){\tiny{$1$}}

\put(40,-5){\tiny{$1$}}

\put(0,63){\tiny{$1$}}

\put(40,63){\tiny{$1$}}

\end{picture}\right) \oplus \fC_N\left(\setlength{\unitlength}{1pt}
\begin{picture}(70,40)(-45,30)

\put(-40,-10){\vector(1,1){20}}

\put(0,-10){\vector(-1,1){20}}

\put(-20,10){\vector(0,1){40}}

\put(-20,50){\vector(1,1){20}}

\put(-20,50){\vector(-1,1){20}}

\put(20,-10){\vector(0,1){80}}

\put(-18,30){\tiny{$2$}}

\put(15,30){\tiny{$1$}}

\put(0,-5){\tiny{$1$}}

\put(-40,-5){\tiny{$1$}}

\put(0,63){\tiny{$1$}}

\put(-40,63){\tiny{$1$}}

\end{picture}\right).
\end{equation}\vspace{3pc}
\end{lemma}

\begin{figure}[ht]
$
\xymatrix{
\setlength{\unitlength}{1pt}
\begin{picture}(60,55)(-30,-15)

\put(-20,0){\vector(0,1){40}}

\put(20,0){\vector(0,1){40}}

\put(-17,20){\tiny{$1$}}

\put(15,20){\tiny{$1$}}

\put(-30,35){\small{$x_1$}}

\put(-30,0){\small{$y_2$}}

\put(23,35){\small{$y_1$}}

\put(23,0){\small{$x_2$}}

\put(-4,-15){$\Gamma_0$}

\end{picture} \ar@<8ex>[rr]^{\chi^0} && \setlength{\unitlength}{1pt}
\begin{picture}(60,55)(-30,-15)

\put(-20,0){\vector(2,1){20}}

\put(20,0){\vector(-2,1){20}}

\put(0,30){\vector(-2,1){20}}

\put(0,30){\vector(2,1){20}}

\put(0,10){\vector(0,1){20}}

\put(-17,7){\tiny{$1$}}

\put(15,7){\tiny{$1$}}

\put(-17,33){\tiny{$1$}}

\put(15,33){\tiny{$1$}}

\put(3,20){\tiny{${2}$}}

\put(-30,35){\small{$x_1$}}

\put(-30,0){\small{$y_2$}}

\put(23,35){\small{$y_1$}}

\put(23,0){\small{$x_2$}}

\put(-4,-15){$\Gamma_1$}

\end{picture} \ar@<-6ex>[ll]^{\chi^1}
}
$
\caption{}\label{def-chi-fig}

\end{figure}

\begin{lemma}\cite[Lemma 3.15]{Wu-triple-trans}\label{lemma-def-chi}
Let $\Gamma_0$ and $\Gamma_1$ be the marked MOY graphs in Figure \ref{def-chi-fig}. Then there exist morphisms of $\zed \oplus \zed^{\oplus 2}$-graded matrix factorizations $\fC_N(\Gamma_0) \xrightarrow{\chi^0} \fC_N(\Gamma_1)\{0,-1\}$ and $\fC_N(\Gamma_1) \xrightarrow{\chi^1} \fC_N(\Gamma_0)\{0,-1\}$ satisfying:
\begin{enumerate}
  \item $\chi^0$ and $\chi^1$ are homotopically non-trivial,
	\item $\chi^1 \circ \chi^0 \simeq (x_2-x_1)\id_{\fC_N(\Gamma_0)}$ and $\chi^0 \circ \chi^1 \simeq (x_2-x_1)\id_{\fC_N(\Gamma_1)}$.
\end{enumerate}
Moreover, up to homotopy and scaling, 
\begin{itemize}
	\item $\chi^0$ is the unique homotopically non-trivial morphism of $\zed \oplus \zed^{\oplus 2}$-graded matrix factorizations from $\fC_N(\Gamma_0)$ to $\fC_N(\Gamma_1)\{0,-1\}$,
	\item $\chi^1$ is the unique homotopically non-trivial morphism of $\zed \oplus \zed^{\oplus 2}$-graded matrix factorizations from $\fC_N(\Gamma_1)$ to $\fC_N(\Gamma_0)\{0,-1\}$.
\end{itemize}
\end{lemma}

\subsection{Definition of $\fH_N$} We first define the chain complex associated to a tangle diagram.

\begin{definition}\label{def-marking-tangle}
Let $T$ be an oriented tangle diagram. We call a segment of $T$ between two adjacent crossings/end points an arc. We color all arcs of $T$ by $1$. A marking of $T$ consists of:
\begin{enumerate}
	\item a collections of marked points on $T$ such that
	      \begin{itemize}
	            \item none of the crossings of $T$ are marked, 
	            \item all end points are marked,
	            \item every arc of $T$ contains at least one marked point,
        \end{itemize}
	\item an assignment of pairwise distinct homogeneous variables of bidegree $(0,2)$ to the marked points such that every marked point is assigned a unique variable.
\end{enumerate} 
\end{definition}

Let $T$ be an oriented tangle with a marking. Recall that $a$ is homogeneous of bidegree $(2,0)$. Denote by 
\begin{itemize}
	\item $R$ the polynomial ring over $\Q$ generated by $a$ and all the variables associated to marked points of $T$, 
	\item $R_\partial$ the polynomial ring over $\Q$ generated by $a$ and all the variables associated to end points of $T$.  
\end{itemize}
Again, we call $R_\partial$ the boundary ring of $T$.

Cut $T$ at all of its marked points. This cuts $T$ into a collection $\{T_1,\dots,T_l\}$ of simple tangles, each of which is of one of the three types in Figure \ref{tangle-pieces-fig} and is marked only at its end points. Denote by $R_i$ the polynomial ring over $\Q$ generated by $a$ and the variables marking end points of $T_i$.

\begin{figure}[ht]
$
\xymatrix{
\setlength{\unitlength}{1pt}
\begin{picture}(60,55)(-30,-15)

\put(0,0){\vector(0,1){40}}

\put(3,20){\tiny{$1$}}

\put(-10,35){\small{$x_1$}}

\put(-10,0){\small{$x_2$}}

\put(-4,-15){$A$}

\end{picture} && \setlength{\unitlength}{1pt}
\begin{picture}(60,55)(-30,-15)

\put(-20,0){\vector(1,1){40}}

\put(20,0){\line(-1,1){18}}

\put(-2,22){\vector(-1,1){18}}

\put(-15,36){\tiny{$1$}}

\put(12,36){\tiny{$1$}}

\put(-15,1){\tiny{$1$}}

\put(12,1){\tiny{$1$}}

\put(-30,35){\small{$x_1$}}

\put(-30,0){\small{$y_2$}}

\put(23,35){\small{$y_1$}}

\put(23,0){\small{$x_2$}}

\put(-4,-15){$C_+$}

\end{picture} && \setlength{\unitlength}{1pt}
\begin{picture}(60,55)(-30,-15)

\put(20,0){\vector(-1,1){40}}

\put(-20,0){\line(1,1){18}}

\put(2,22){\vector(1,1){18}}

\put(-15,36){\tiny{$1$}}

\put(12,36){\tiny{$1$}}

\put(-15,1){\tiny{$1$}}

\put(12,1){\tiny{$1$}}

\put(-30,35){\small{$x_1$}}

\put(-30,0){\small{$y_2$}}

\put(23,35){\small{$y_1$}}

\put(23,0){\small{$x_2$}}

\put(-4,-15){$C_-$}

\end{picture} 
}
$
\caption{}\label{tangle-pieces-fig}

\end{figure}

If $T_i=A$, then $R_i =\Q[a,x_1,x_2]$ and $\fC_N(T_i)$ is the chain complex over $\hmf_{R_i,a(x_1^{N+1}-x_2^{N+1})}$ given by
\begin{equation}\label{eq-def-chain-arc}
\fC_N(T_i)= 0 \rightarrow \underbrace{\fC_N(A)}_{0} \rightarrow 0, 
\end{equation}
where the $\fC_N(A)$ on the right hand side is the matrix factorization associated to the MOY graph $A$, and the under-brace indicates the homological grading.

\begin{figure}[ht]
$
\xymatrix{
&& \setlength{\unitlength}{1pt}
\begin{picture}(60,55)(-30,-15)

\put(-20,0){\vector(1,1){40}}

\put(20,0){\line(-1,1){18}}

\put(-2,22){\vector(-1,1){18}}

\put(-15,36){\tiny{$1$}}

\put(12,36){\tiny{$1$}}

\put(-15,1){\tiny{$1$}}

\put(12,1){\tiny{$1$}}

\put(-30,35){\small{$x_1$}}

\put(-30,0){\small{$y_2$}}

\put(23,35){\small{$y_1$}}

\put(23,0){\small{$x_2$}}

\put(-4,-15){$C_+$}

\end{picture}  \ar@<-12ex>[lld]_{0}\ar@<12ex>[rrd]^{+1} && \\
 \setlength{\unitlength}{1pt}
\begin{picture}(60,55)(-30,-15)

\put(-20,0){\vector(0,1){40}}

\put(20,0){\vector(0,1){40}}

\put(-17,20){\tiny{$1$}}

\put(15,20){\tiny{$1$}}

\put(-30,35){\small{$x_1$}}

\put(-30,0){\small{$y_2$}}

\put(23,35){\small{$y_1$}}

\put(23,0){\small{$x_2$}}

\put(-4,-15){$\Gamma_0$}

\end{picture}&&&& \setlength{\unitlength}{1pt}
\begin{picture}(60,55)(-30,-15)

\put(-20,0){\vector(2,1){20}}

\put(20,0){\vector(-2,1){20}}

\put(0,30){\vector(-2,1){20}}

\put(0,30){\vector(2,1){20}}

\put(0,10){\vector(0,1){20}}

\put(-17,7){\tiny{$1$}}

\put(15,7){\tiny{$1$}}

\put(-17,33){\tiny{$1$}}

\put(15,33){\tiny{$1$}}

\put(3,20){\tiny{${2}$}}

\put(-30,35){\small{$x_1$}}

\put(-30,0){\small{$y_2$}}

\put(23,35){\small{$y_1$}}

\put(23,0){\small{$x_2$}}

\put(-4,-15){$\Gamma_1$}

\end{picture} \\
&& \setlength{\unitlength}{1pt}
\begin{picture}(60,55)(-30,-15)

\put(20,0){\vector(-1,1){40}}

\put(-20,0){\line(1,1){18}}

\put(2,22){\vector(1,1){18}}

\put(-15,36){\tiny{$1$}}

\put(12,36){\tiny{$1$}}

\put(-15,1){\tiny{$1$}}

\put(12,1){\tiny{$1$}}

\put(-30,35){\small{$x_1$}}

\put(-30,0){\small{$y_2$}}

\put(23,35){\small{$y_1$}}

\put(23,0){\small{$x_2$}}

\put(-4,-15){$C_-$}

\end{picture} \ar[llu]_{0} \ar[rru]^{-1} &&
}
$
\caption{}\label{crossing-res-fig}

\end{figure}

If $T_i=C_\pm$, then $R_i = \Q[a,x_1,x_2,y_1,y_2]$ and $\fC_N(T_i)$ is the chain complex over $\hmf_{R_i,a(x_1^{N+1}+y_1^{N+1}-x_2^{N+1}-y_2^{N+1})}$ given by
\begin{eqnarray}
\label{eq-def-chain-crossing+} \fC_N(C_+) & = & 0 \rightarrow \underbrace{\fC_N(\Gamma_1)\left\langle 1\right\rangle\{1,N\}}_{-1} \xrightarrow{\chi^1} \underbrace{\fC_N(\Gamma_0)\left\langle 1\right\rangle\{1,N-1\}}_{0} \rightarrow 0, \\
\label{eq-def-chain-crossing-} \fC_N(C_-) & = & 0 \rightarrow \underbrace{\fC_N(\Gamma_0)\left\langle 1\right\rangle\{-1,-N+1\}}_{0} \xrightarrow{\chi^0} \underbrace{\fC_N(\Gamma_1)\left\langle 1\right\rangle\{-1,-N\}}_{1} \rightarrow 0,
\end{eqnarray}
where $\Gamma_0$ and $\Gamma_1$ are the resolutions of $C_\pm$ given in Figure \ref{crossing-res-fig}, the morphisms $\chi^0$ and $\chi^1$ are defined in Lemma \ref{lemma-def-chi} and the under-braces indicate the homological gradings.

Note that, in all three cases, the differential map of $\fC_N(T_i)$ consists of homogeneous morphisms of matrix factorizations preserving the $\zed_2 \oplus \zed^{\oplus 2}$-grading. Of course, this differential map raises the homological grading by $1$.

\begin{definition}\label{def-chain-tangle}
We define the chain complex $\fC_N(T)$ associated to $T$ to be 
\[
\fC_N(T) := \bigotimes_{i=1}^{l} (\fC_N(T_i)\otimes_{R_i} R),
\]
where the big tensor product ``$\bigotimes_{i=1}^{l}$" is taken over $R$. We view $\fC_N(T)$ as a chain complex of $\zed_2\oplus \zed^{\oplus 2}$-graded matrix factorizations over the ring $R_\partial$.

$\fC_N(T)$ is equipped with a $\zed_2\oplus \zed^{\oplus 3}$-grading, where the $\zed_2\oplus \zed^{\oplus 2}$-grading comes from the underlying matrix factorization and the additional $\zed$-grading is the homological grading.

Note that, if $T$ is an oriented link diagram, then $\fC_N(T)$ is a chain complex over the category $\hmf^{\mathrm{all}}_{\Q[a],0}$.
\end{definition}

\begin{lemma}\cite[Lemma 4.5, and Propositions 5.5, 6.1, 7.5]{Wu-triple-trans}\label{lemma-chain-inv}
The homotopy type of $\fC_N(T)$ is independent of the marking of $T$ and invariant under positive Reidemeister move I and braid-like Reidemeister moves II and III.
\end{lemma}

Now let $L$ be a link diagram with a marking. Note $\fC_N(L)$ has two differential maps:
\begin{enumerate}
	\item The differential $d_{mf}$ of the underlying matrix factorization structure of $\fC_N(L)$. 
  \item The differential $d_\chi$ from the crossing information given in equations \eqref{eq-def-chain-arc}, \eqref{eq-def-chain-crossing+} and \eqref{eq-def-chain-crossing-}.
\end{enumerate}
As a matrix factorization, $\fC_N(L)$ is a matrix factorization of $0$. So $d_{mf}^2=0$. Thus, the homology $H(\fC_N(L), d_{mf})$ is well defined. In fact, $H(\fC_N(L), d_{mf})$ inherits the $\zed_2\oplus \zed^{\oplus 3}$-grading of $\fC_N(L)$ and $(H(\fC_N(L), d_{mf}),d_\chi)$ is a chain complex with a homological $\zed$-grading of $\zed_2\oplus \zed^{\oplus 2}$-graded $\Q[a]$-modules.

\begin{definition}\label{def-homology-link}
$\fH_N(L) := H(H(\fC_N(L), d_{mf}),d_\chi)$. It is a $\zed_2\oplus \zed^{\oplus 3}$-graded $\Q[a]$-module.
\end{definition}

As a simple corollary of Lemma \ref{lemma-chain-inv}, we have:

\begin{corollary}\label{cor-homology-inv}
The $\zed_2\oplus \zed^{\oplus 3}$-graded $\Q[a]$-module $\fH_N(L)$ is independent of the marking of $L$ and invariant under positive Reidemeister move I and braid-like Reidemeister moves II and III.
\end{corollary}

Clearly, Theorem \ref{transversal-markov} follows from Lemma \ref{lemma-chain-inv} and Corollary \ref{cor-homology-inv}.

\subsection{The $\slmf(N)$ Khovanov-Rozansky homology $H_N$}\label{subsec-KR-sl-N} If we set $a=1$ in the above construction, then we get the $\slmf(N)$ Khovanov-Rozansky homology $H_N$ define in \cite{KR1}. More precisely, for any tangle $T$, let 
\begin{equation}\label{eq-def-KR-chain}
C_N(T) = \fC_N(T)/(a-1)\fC_N(T).
\end{equation}
Then $C_N(T)$ is the $\slmf(N)$ Khovanov-Rozansky chain complex defined in \cite{KR1}. Note that $C_N(T)$ inherits the $\zed_2$-, homological and $x$-gradings of $\fC_N(T)$. It also inherits the differentials $d_{mf}$ and $d_\chi$. For a link diagram $L$, 
\begin{equation}\label{eq-def-KR-homology}
H_N(L) = H((C_N(L),d_{mf}),d_\chi)
\end{equation}
is the $\slmf(N)$ Khovanov-Rozansky homology defined in \cite{KR1}. $H_N(L)$ inherits the gradings of $C_N(L)$ and is a $\zed_2\oplus \zed^{\oplus 2}$-graded $\Q$-linear space.

The $\zed_2 \oplus \zed$-graded matrix factorization $C_N(\Gamma) = \fC_N(\Gamma)/(a-1)\fC_N(\Gamma)$ of a MOY graph $\Gamma$ satisfies decompositions similar to those in Lemma \ref{lemma-MOY-decomps}.

\begin{lemma}\label{lemma-MOY-decomps-KR}\cite{KR1}
As matrix factorizations over the respective boundary rings, we have: \vspace{1pc}
\begin{equation}\label{eq-MOY-I-KR}
C_N\left(\setlength{\unitlength}{1pt}
\begin{picture}(61,20)(-31,20)

\put(0,10){\vector(0,1){20}}

\put(-30,10){\line(0,1){20}}

\qbezier(0,30)(0,40)(-20,40)

\qbezier(-30,30)(-30,40)(-20,40)

\qbezier(0,10)(0,0)(-20,0)

\qbezier(-30,10)(-30,0)(-20,0)

\put(20,0){\vector(0,1){40}}

\put(3,20){\tiny{$1$}}

\put(15,20){\tiny{$1$}}

\end{picture}\right) \simeq C_N\left(\setlength{\unitlength}{1pt}
\begin{picture}(50,20)(-25,20)

\put(-10,5){\line(2,1){10}}

\put(20,0){\vector(-2,1){20}}

\put(0,30){\vector(-2,1){10}}

\put(0,30){\vector(2,1){20}}

\put(0,10){\vector(0,1){20}}

\qbezier(-10,35)(-20,40)(-20,30)

\qbezier(-10,5)(-20,0)(-20,10)

\put(-20,10){\line(0,1){20}}

\put(15,7){\tiny{$1$}}

\put(-17,30){\tiny{$1$}}

\put(15,33){\tiny{$1$}}

\put(3,20){\tiny{${2}$}}

\end{picture}\right)\{1\}_x \oplus C_N\left(\setlength{\unitlength}{1pt}
\begin{picture}(10,20)(15,20)

\put(20,0){\vector(0,1){40}}

\put(15,20){\tiny{$1$}}

\end{picture}\right)\left\langle 1\right\rangle\{1-N\}_x,
\end{equation}

\begin{equation}\label{eq-MOY-II-KR}
C_N\left(\setlength{\unitlength}{1pt}
\begin{picture}(32,35)(-77,30)

\put(-60,5){\vector(0,1){10}}

\qbezier(-60,15)(-70,15)(-70,25)

\put(-70,25){\vector(0,1){10}}

\qbezier(-70,35)(-70,45)(-60,45)

\qbezier(-60,15)(-50,15)(-50,25)

\put(-50,25){\vector(0,1){10}}

\qbezier(-50,35)(-50,45)(-60,45)

\put(-60,45){\vector(0,1){10}}

\put(-65,48){\tiny{$2$}}

\put(-65,7){\tiny{$2$}}

\put(-75,30){\tiny{$1$}}

\put(-49,30){\tiny{$1$}}

\put(-70,-5){\vector(1,1){10}}

\put(-50,-5){\vector(-1,1){10}}

\put(-60,55){\vector(1,1){10}}

\put(-60,55){\vector(-1,1){10}}

\put(-75,-5){\tiny{$1$}}

\put(-75,60){\tiny{$1$}}

\put(-49,-5){\tiny{$1$}}

\put(-49,60){\tiny{$1$}}

\end{picture}\right) \simeq C_N\left(\setlength{\unitlength}{1pt}
\begin{picture}(32,35)(-77,30)

\put(-60,5){\vector(0,1){50}}

\put(-65,30){\tiny{$2$}}

\put(-70,-5){\vector(1,1){10}}

\put(-50,-5){\vector(-1,1){10}}

\put(-60,55){\vector(1,1){10}}

\put(-60,55){\vector(-1,1){10}}

\put(-75,-5){\tiny{$1$}}

\put(-75,60){\tiny{$1$}}

\put(-49,-5){\tiny{$1$}}

\put(-49,60){\tiny{$1$}}

\end{picture}\right)\{-1\}_x \oplus C_N\left(\setlength{\unitlength}{1pt}
\begin{picture}(32,35)(-77,30)

\put(-60,5){\vector(0,1){50}}

\put(-65,30){\tiny{$2$}}

\put(-70,-5){\vector(1,1){10}}

\put(-50,-5){\vector(-1,1){10}}

\put(-60,55){\vector(1,1){10}}

\put(-60,55){\vector(-1,1){10}}

\put(-75,-5){\tiny{$1$}}

\put(-75,60){\tiny{$1$}}

\put(-49,-5){\tiny{$1$}}

\put(-49,60){\tiny{$1$}}

\end{picture}\right)\{1\}_x,
\end{equation}

\begin{equation}\label{eq-MOY-III-KR}
C_N\left(\setlength{\unitlength}{1pt}
\begin{picture}(70,40)(-45,30)

\put(-40,-10){\vector(1,1){20}}

\put(0,-10){\vector(-1,1){20}}

\put(-20,10){\vector(0,1){10}}

\put(-20,20){\vector(0,1){20}}

\put(-20,40){\vector(0,1){10}}

\put(-20,50){\vector(1,1){20}}

\put(-20,50){\vector(-1,1){20}}

\put(20,-10){\vector(0,1){30}}

\put(20,20){\vector(0,1){20}}

\put(20,40){\vector(0,1){30}}

\put(-20,20){\vector(1,0){40}}

\put(20,40){\vector(-1,0){40}}

\put(-18,43){\tiny{$2$}}

\put(-18,30){\tiny{$1$}}

\put(-18,12){\tiny{$2$}}

\put(-1,33){\tiny{$1$}}

\put(-1,22){\tiny{$1$}}

\put(15,50){\tiny{$1$}}

\put(15,30){\tiny{$2$}}

\put(15,-5){\tiny{$1$}}

\put(0,-5){\tiny{$1$}}

\put(-40,-5){\tiny{$1$}}

\put(0,63){\tiny{$1$}}

\put(-40,63){\tiny{$1$}}

\end{picture}\right) \oplus C_N\left(\setlength{\unitlength}{1pt}
\begin{picture}(70,40)(-25,30)

\put(40,-10){\vector(-1,1){20}}

\put(0,-10){\vector(1,1){20}}

\put(-20,-10){\vector(0,1){80}}

\put(20,50){\vector(1,1){20}}

\put(20,50){\vector(-1,1){20}}

\put(20,10){\vector(0,1){40}}

\put(-18,30){\tiny{$1$}}

\put(15,30){\tiny{$2$}}

\put(0,-5){\tiny{$1$}}

\put(40,-5){\tiny{$1$}}

\put(0,63){\tiny{$1$}}

\put(40,63){\tiny{$1$}}

\end{picture}\right) \simeq C_N\left(\setlength{\unitlength}{1pt}
\begin{picture}(70,40)(-25,30)

\put(40,-10){\vector(-1,1){20}}

\put(0,-10){\vector(1,1){20}}

\put(-20,-10){\vector(0,1){30}}

\put(-20,20){\vector(0,1){20}}

\put(-20,40){\vector(0,1){30}}

\put(20,50){\vector(1,1){20}}

\put(20,50){\vector(-1,1){20}}

\put(20,10){\vector(0,1){10}}

\put(20,20){\vector(0,1){20}}

\put(20,40){\vector(0,1){10}}

\put(20,20){\vector(-1,0){40}}

\put(-20,40){\vector(1,0){40}}

\put(15,43){\tiny{$2$}}

\put(-18,30){\tiny{$2$}}

\put(15,12){\tiny{$2$}}

\put(-1,33){\tiny{$1$}}

\put(-1,22){\tiny{$1$}}

\put(-18,50){\tiny{$1$}}

\put(15,30){\tiny{$1$}}

\put(-18,-5){\tiny{$1$}}

\put(0,-5){\tiny{$1$}}

\put(40,-5){\tiny{$1$}}

\put(0,63){\tiny{$1$}}

\put(40,63){\tiny{$1$}}

\end{picture}\right) \oplus C_N\left(\setlength{\unitlength}{1pt}
\begin{picture}(70,40)(-45,30)

\put(-40,-10){\vector(1,1){20}}

\put(0,-10){\vector(-1,1){20}}

\put(-20,10){\vector(0,1){40}}

\put(-20,50){\vector(1,1){20}}

\put(-20,50){\vector(-1,1){20}}

\put(20,-10){\vector(0,1){80}}

\put(-18,30){\tiny{$2$}}

\put(15,30){\tiny{$1$}}

\put(0,-5){\tiny{$1$}}

\put(-40,-5){\tiny{$1$}}

\put(0,63){\tiny{$1$}}

\put(-40,63){\tiny{$1$}}

\end{picture}\right).
\end{equation}\vspace{3pc}

In the above, $\{\ast\}_x$ means shifting the $x$-grading by $\ast$.
\end{lemma}

The following invariance theorem for $H_N$ is established in \cite{KR1}.

\begin{theorem}\cite{KR1}
The homotopy type of $C_N(T)$, including its $\zed_2\oplus \zed^{\oplus 2}$-grading, is independent of markings and invariant under all Reidemeister moves. Consequently, every Reidemeister move on $L$ induces an isomorphism of $H_N(L)$ preserving its $\zed_2\oplus \zed^{\oplus 2}$-graded $\Q$-linear space structure.
\end{theorem}

\section{Graded Module Structure of $\fH_N$}\label{sec-fH-module}

In this section, we study the $\zed_2\oplus \zed^{\oplus 3}$-graded $\Q[a]$-module of $\fH_N$. The goal is to prove Theorem \ref{thm-fH-module}.

\subsection{Resolved braids} In this subsection, we review some basic properties of resolved braids introduced in \cite{Wu5}.

\begin{figure}[ht]
\[
\setlength{\unitlength}{1pt}
\begin{picture}(100,45)(-50,-15)

\put(-50,30){\vector(0,-1){30}}

\put(-40,15){$\dots$}

\put(-20,30){\vector(0,-1){30}}

\put(-10,30){\vector(1,-1){10}}

\put(10,30){\vector(-1,-1){10}}

\put(0,20){\vector(0,-1){10}}

\put(0,10){\vector(1,-1){10}}

\put(0,10){\vector(-1,-1){10}}

\put(50,30){\vector(0,-1){30}}

\put(30,15){$\dots$}

\put(20,30){\vector(0,-1){30}}

\put(-15,25){\tiny{$1$}}

\put(12,25){\tiny{$1$}}

\put(-15,0){\tiny{$1$}}

\put(12,0){\tiny{$1$}}

\put(2,13){\tiny{$2$}}

\put(-4,-15){$\tau_i$}

\put(-48,25){\tiny{$1$}}

\put(-25,25){\tiny{$1$}}

\put(22,25){\tiny{$1$}}

\put(45,25){\tiny{$1$}}

\end{picture}
\]
\caption{}\label{fig-tau-i}

\end{figure}

\begin{definition}\label{def-resolved-braids}
For positive integers $b$, $i$ with $1\leq i\leq b-1$, let $\tau_i$ be the MOY graph depicted in Figure \ref{fig-tau-i}. That is, from left to right, $\tau_i$ consists of $i-1$ downward $1$-colored edges, then a downward $2$-colored edge with two $1$-colored edges entering through the top and two $1$-colored edges exiting through the bottom, and then $b-i-1$ more downward $1$-colored edges. 

We use $(\tau_{i_1}\cdots\tau_{i_m})_b$ to represent the MOY graph formed by stacking the graphs $\tau_{i_1},\dots,\tau_{i_m}$ together vertically from top to bottom with the bottom end points of $\tau_{i_l}$ identified with the corresponding top end points of $\tau_{i_{l+1}}$. We call $(\tau_{i_1}\cdots\tau_{i_m})_b$ a resolved braid of $b$-strands. If the number of strands is clear from the context, then we drop the lower index $b$ and simply write $\tau_{i_1}\cdots\tau_{i_m}$.

Denote by $\overline{(\tau_{i_1}\cdots\tau_{i_m})_b}$ the closed MOY graph obtained from $(\tau_{i_1}\cdots\tau_{i_m})_b$ by attaching a $1$-colored edge from each end point at the bottom to the corresponding end point at the top. We call $\overline{(\tau_{i_1}\cdots\tau_{i_m})_b}$ a closed resolved braid of $b$-strands. Again, if the number of strands is clear from the context, then we drop the lower index $b$ and simply write $\overline{\tau_{i_1}\cdots\tau_{i_m}}$.

We use $(\emptyset)_b$ to represent $b$ vertical downward $1$-colored edges, and, therefore, $\overline{(\emptyset)_b}$ represents $b$ concentric $1$-colored circles. Again, if the number of strands is clear from the context, then we drop the lower index $b$.
\end{definition}

\begin{figure}[ht]
\[
\xymatrix{
\setlength{\unitlength}{1pt}
\begin{picture}(60,145)(-30,-35)

\put(-30,90){\vector(1,-1){10}}

\put(-10,90){\vector(-1,-1){10}}

\put(-20,80){\vector(0,-1){10}}

\put(-20,70){\line(-1,-1){10}}

\put(-30,60){\vector(0,-1){60}}

\put(-20,70){\vector(1,-1){20}}

\put(10,60){\vector(-1,-1){10}}

\put(10,90){\line(0,-1){30}}

\put(0,50){\vector(0,-1){10}}

\put(0,40){\line(-1,-1){10}}

\put(-10,30){\vector(0,-1){30}}

\put(0,40){\vector(1,-1){20}}

\put(30,90){\line(0,-1){60}}

\put(30,30){\vector(-1,-1){10}}

\put(20,20){\vector(0,-1){10}}

\put(20,10){\vector(-1,-1){10}}

\put(20,10){\vector(1,-1){10}}

\put(-32,85){\tiny{$1$}}

\put(-28,5){\tiny{$1$}}

\put(-8,85){\tiny{$1$}}

\put(-8,62){\tiny{$1$}}

\put(-8,5){\tiny{$1$}}

\put(7,5){\tiny{$1$}}

\put(7,33){\tiny{$1$}}

\put(7,85){\tiny{$1$}}

\put(27,4){\tiny{$1$}}

\put(27,85){\tiny{$1$}}

\put(-18,73){\tiny{$2$}}

\put(2,43){\tiny{$2$}}

\put(22,13){\tiny{$2$}}

\put(-10,-35){$(\tau_1\tau_2\tau_3)_4$}

\end{picture} && \setlength{\unitlength}{1pt}
\begin{picture}(140,145)(-30,-35)

\put(-30,90){\vector(1,-1){10}}

\put(-10,90){\vector(-1,-1){10}}

\put(-20,80){\vector(0,-1){10}}

\put(-20,70){\line(-1,-1){10}}

\put(-30,60){\vector(0,-1){60}}

\put(-20,70){\vector(1,-1){20}}

\put(10,60){\vector(-1,-1){10}}

\put(10,90){\line(0,-1){30}}

\put(0,50){\vector(0,-1){10}}

\put(0,40){\line(-1,-1){10}}

\put(-10,30){\vector(0,-1){30}}

\put(0,40){\vector(1,-1){20}}

\put(30,90){\line(0,-1){60}}

\put(30,30){\vector(-1,-1){10}}

\put(20,20){\vector(0,-1){10}}

\put(20,10){\vector(-1,-1){10}}

\put(20,10){\vector(1,-1){10}}

\put(-32,85){\tiny{$1$}}

\put(-28,5){\tiny{$1$}}

\put(-8,85){\tiny{$1$}}

\put(-8,62){\tiny{$1$}}

\put(-8,5){\tiny{$1$}}

\put(7,5){\tiny{$1$}}

\put(7,33){\tiny{$1$}}

\put(7,85){\tiny{$1$}}

\put(27,4){\tiny{$1$}}

\put(27,85){\tiny{$1$}}

\put(-18,73){\tiny{$2$}}

\put(2,43){\tiny{$2$}}

\put(22,13){\tiny{$2$}}

\put(30,0){\line(1,0){20}}

\put(30,90){\line(1,0){20}}

\put(50,0){\line(0,1){90}}

\put(10,0){\line(0,-1){5}}

\put(10,90){\line(0,1){5}}

\put(10,-5){\line(1,0){60}}

\put(10,95){\line(1,0){60}}

\put(70,-5){\line(0,1){100}}

\put(-10,0){\line(0,-1){10}}

\put(-10,90){\line(0,1){10}}

\put(-10,-10){\line(1,0){100}}

\put(-10,100){\line(1,0){100}}

\put(90,-10){\line(0,1){110}}

\put(-30,0){\line(0,-1){15}}

\put(-30,90){\line(0,1){15}}

\put(-30,-15){\line(1,0){140}}

\put(-30,105){\line(1,0){140}}

\put(110,-15){\line(0,1){120}}

\put(30,-35){$\overline{(\tau_1\tau_2\tau_3)_4}$}

\end{picture}
}
\]
\caption{}\label{fig-resolved-braid-example}

\end{figure}

\begin{remark}\label{remark-resolved-braids}
\begin{enumerate}
	\item Comparing Definition \ref{def-resolved-braids} to the resolutions in Figure \ref{crossing-res-fig}, one can see that, if we choose a resolution for every crossing in a (closed) braid, then we get a (closed) resolved braid as defined in Definition \ref{def-resolved-braids}.
	\item  There are two obvious types of isotopies of resolved braids and closed resolved braids:
     \begin{itemize}
	      \item[I$_1$:] If $|i-j|>1$, then $\tau_i\tau_j$ is isotopic to $\tau_j\tau_i$;
				\item[I$_2$:] If $\mu$ and $\nu$ are two words in $\tau_{1},\dots,\tau_{b-1}$, then $\overline{\mu\nu}$ is isotopic to
$\overline{\nu\mu}$.
     \end{itemize}
\end{enumerate}
\end{remark}

\begin{definition}\label{def-weight}
We define the weight of the closed resolved braid $\overline{\tau_{i_1}\cdots\tau_{i_m}}$ to be $w(\overline{\tau_{i_1}\cdots\tau_{i_m}}) = i_1 + \cdots + i_m$.
\end{definition}

In \cite{Wu5}, the author introduced a scheme to perform inductive arguments on the weights of closed resolved braids using the decompositions in Lemma \ref{lemma-MOY-decomps-KR}. The key to this scheme is Corollary \ref{cor-resolved-braid-induction} below, which is a simple consequence of Lemma \ref{lemma-resolved-braid-I-1}.

\begin{lemma}\cite[Lemma 3.5]{Wu5}\label{lemma-resolved-braid-I-1}
Let $\mu=\tau_{i_1}\cdots\tau_{i_m}$ be a resolved braid with $b$ strands satisfying: 
\begin{itemize}
	\item $m\geq2$,
	\item $i_1=i_m=i$,
	\item $i_l<i$ for $1<l<m$.
\end{itemize}
Then, via a finite sequence of isotopies of type I$_1$, $\mu$ is isotopic to a resolved braid $\mu'$ that contains a segment of the form $\tau_j\tau_j$ or $\tau_j\tau_{j-1}\tau_j$ for some $j\leq i$.
\end{lemma}

\begin{corollary}\label{cor-resolved-braid-induction}
Let $\overline{\mu}$ be a closed resolved braids with $b$ strands. Then, via a finite sequence of isotopies of types I$_1$ and I$_2$, $\overline{\mu}$ is isotopic to a closed resolved braid of one of the following three types:
\begin{enumerate}[(a)]
	\item $\overline{\tau_{i_1}\cdots\tau_{i_m}\tau_i}$, where $i>i_1,\dots,i_m$;
	\item $\overline{\tau_{i_1}\cdots\tau_{i_m}\tau_j\tau_j}$;
	\item $\overline{\tau_{i_1}\cdots\tau_{i_m}\tau_j\tau_{j-1}\tau_j}$.
\end{enumerate}
\end{corollary}

\subsection{Homology of closed resolve braids} In this subsection, we study the $\Q[a]$-module structure of the homology of closed resolved braids. The goal is to establish Lemma \ref{lemma-module-homology-resolved-braid} below.

\begin{lemma}\label{lemma-homology-empty-braid}
Let $\overline{(\emptyset})_b$ be the closed resolved braid with $b$-strands corresponding to the empty word, that is, the MOY graph consisting of $b$ concentric $1$-colored circles. Define the $\zed_2 \oplus \zed^{\oplus 2}$-graded $\Q[a]$-modules $\mathcal{M}_0$, $\mathcal{M}_1$ and $\mathcal{M}_\infty$ by
\begin{eqnarray*}
\mathcal{M}_0 & := & \Q[a] \left\langle 1 \right\rangle \{-1, 1-N\} \oplus \Q[a], \\
\mathcal{M}_1 & := & \bigoplus_{l=0}^{N-1} \Q[a] \left\langle 1 \right\rangle \{-1, 1-N+2l\}, \\
\mathcal{M}_\infty & := & \bigoplus_{l=N}^{\infty} \Q[a]/(a) \left\langle 1 \right\rangle \{-1, 1-N+2l\}.
\end{eqnarray*}
Then, as a $\zed_2 \oplus \zed^{\oplus 2}$-graded $\Q[a]$-module,
\[
\fH_N(\overline{(\emptyset})_b) \cong \mathcal{M}_1 ^{\otimes b} \oplus \left( \bigoplus_{j=0}^{b-1} \mathcal{M}_0^{\otimes j} \otimes \mathcal{M}_1^{\otimes (b-1-j)} \right) \otimes \mathcal{M}_\infty,
\]
where all the tensor products are over $\Q[a]$.
\end{lemma}

\begin{proof}
We prove this lemma by an induction on $b$. Mark $\overline{(\emptyset})_1$ by a single variable $x$. Then 
\[
\fC_N(\overline{(\emptyset)_1}) = ((N+1)ax_1^{N},0)_{\Q[a,x]} = \Q[a,x] \xrightarrow{(N+1)ax^{N}} \Q[a,x] \{-1,1-N\} \xrightarrow{0} \Q[a,x].
\]
So $\fH_N(\overline{(\emptyset})_1) \cong \Q[a,x]/(ax^N) \left\langle 1\right\rangle \{-1, 1-N\} \cong \mathcal{M}_1 \oplus \mathcal{M}_\infty$. This proves the lemma for $b=1$. 

Now assume the lemma is true for $\overline{(\emptyset)_{b-1}}$. Consider $\overline{(\emptyset)_b}$. Mark the $j$th circle in $\overline{(\emptyset)_b}$ by a single variable $x_j$. Then 
\[
\fC_N(\overline{(\emptyset)_b}) = \left(%
\begin{array}{cc}
  (N+1)ax_1^N & 0 \\
  (N+1)ax_2^N & 0 \\
  \dots & \dots \\
  (N+1)ax_b^N & 0
\end{array}%
\right)_{\Q[a,x_1,x_2,\dots,x_b]}.
\]
Thus, by Proposition \ref{prop-contraction-weak}, $\fC_N(\overline{(\emptyset)_b})$ is quasi-isomorphic to 
\[
\left(%
\begin{array}{cc}
  (N+1)ax_1^N & 0 \\
  (N+1)ax_2^N & 0 \\
  \dots & \dots \\
  (N+1)ax_{b-1}^N & 0
\end{array}%
\right)_{\Q[a,x_1,x_2,\dots,x_{b-1}]} \otimes_{\Q[a]} \Q[a,x_b]/(ax_b^N)\left\langle 1\right\rangle \{-1, 1-N\}.
\] 
Note that:
\begin{enumerate}
	\item $\fC_N(\overline{(\emptyset)_{b-1}}) \cong \left(%
\begin{array}{cc}
  (N+1)ax_1^N & 0 \\
  (N+1)ax_2^N & 0 \\
  \dots & \dots \\
  (N+1)ax_{b-1}^N & 0
\end{array}%
\right)_{\Q[a,x_1,x_2,\dots,x_{b-1}]}$,
  \item $\Q[a,x_b]/(ax_b^N)\left\langle 1\right\rangle \{-1, 1-N\} \cong \mathcal{M}_1 \oplus \mathcal{M}_\infty$,
	\item $\mathcal{M}_1$ is a free $\Q[a]$-module,
	\item The homology of $\left(%
\begin{array}{cc}
  (N+1)ax_1^N & 0 \\
  (N+1)ax_2^N & 0 \\
  \dots & \dots \\
  (N+1)ax_{b-1}^N & 0
\end{array}%
\right)_{\Q[a,x_1,x_2,\dots,x_{b-1}]} \otimes_{\Q[a]} \mathcal{M}_\infty$ is isomorphic to $\mathcal{M}_0^{\otimes (b-1)} \otimes \mathcal{M}_\infty$.
\end{enumerate}
Putting the above together, we get
\[
\fH_N(\overline{(\emptyset)_b}) \cong \fH_N(\overline{(\emptyset)_{b-1}}) \otimes_{\Q[a]} \mathcal{M}_1 \oplus \mathcal{M}_0^{\otimes (b-1)} \otimes \mathcal{M}_\infty.
\]
This isomorphism and the assumption that the lemma is true for $\overline{(\emptyset)_{b-1}}$ imply that the lemma is true for $\overline{(\emptyset)_b}$.
\end{proof}

To discuss the homology of a general closed resolved braid, we need the following lemma, which is a slight refinement of the usual structure theorem of modules over a principal deal domain.

\begin{lemma}\cite[Lemma 9.2]{Wu-triple-trans}\label{lemma-fg-graded-module-structure}
Suppose that $M$ is a finitely generated $\zed$-graded $\Q[a]$-module. Then, as a $\zed$-graded $\Q[a]$-module, $M\cong (\bigoplus_{j=1}^m \Q[a]\{s_j\}_a) \bigoplus (\bigoplus_{k=1}^n \Q[a]/(a^{l_k})\{t_k\}_a)$, where $\{\ast\}_a$ means shifting the $a$-grading by $\ast$, and the sequences $\{s_1,\dots,s_m\} \subset \zed$, $\{(l_1,t_1),\dots,(l_n,t_n)\} \subset \zed^{\oplus 2}$ are uniquely determined by $M$ up to permutation. We call this decomposition the standard decomposition of $M$.
\end{lemma}

\begin{definition}\label{def-MOY-homology-components}
For a closed resolved braid $\overline{\mu}$, we denote by $\fH_N^{\ve,j,k}(\overline{\mu})$ (resp. $\fC_N^{\ve,j,k}(\overline{\mu})$) the homogeneous component of  $\fH_N(\overline{\mu})$ (resp. $\fC_N(\overline{\mu})$) of $\zed_2$-degree $\ve$, $a$-degree $j$ and $x$-degree $k$. If we replace one of these indices by a $\star$, it means we direct sum the components over all possible values of that index. For example, $\fH_N^{\ve,\star,k}(\overline{\mu}) = \bigoplus_{j\in\zed}\fH_N^{\ve,j,k}(\overline{\mu})$.

Similarly, we denote by $H_N^{\ve,k}(\overline{\mu})$ (resp. $C_N^{\ve,k}(\overline{\mu})$) the homogeneous component  of $\zed_2$-degree $\ve$ and $x$-degree $k$ of the $\slmf(N)$ Khovanov-Rozansky homology $H_N(\overline{\mu})$ (resp. $C_N(\overline{\mu})$) of $\overline{\mu}$. 
\end{definition}

\begin{lemma}\label{lemma-module-homology-resolved-braid}
For a closed resolved braid $\overline{(\tau_{i_1}\cdots\tau_{i_m})_b}$ of $b$ strands, we have that, as a $\zed$-graded $\Q[a]$-module,
\[
\fH_N^{\ve,\star,k}(\overline{(\tau_{i_1}\cdots\tau_{i_m})_b}) \cong H_N^{\ve,k}(\overline{(\tau_{i_1}\cdots\tau_{i_m})_b}) \otimes_\Q \Q[a]\{-b\}_a \oplus \bigoplus_{i=1}^l \Q[a]/(a) \{s_i\}_a,
\]
where
\begin{itemize}
	\item we give $H_N^{\ve,k}(\overline{(\tau_{i_1}\cdots\tau_{i_m})_b})$ the $a$-grading $0$, and $\{\ast\}_a$ means shifting the $a$-grading by $\ast$,
	\item up to permutation, the sequence $\{s_1,\dots,s_l\}$ is uniquely determined by $\overline{(\tau_{i_1}\cdots\tau_{i_m})_b}$, $N$, $k$ and $\ve$,
	\item $-b \leq s_i \leq -1$ and $(N-1)s_i \leq k -2N+m$ for $i=1,\dots,l$.
\end{itemize}
\end{lemma}

\begin{proof}
From the construction of $\fC_N(\overline{(\tau_{i_1}\cdots\tau_{i_m})_b})$, one can see that $\fC_N^{\ve,\star,k}(\overline{(\tau_{i_1}\cdots\tau_{i_m})_b})$ is a finitely generated free $\Q[a]$-module. This implies that $\fH_N^{\ve,\star,k}(\overline{(\tau_{i_1}\cdots\tau_{i_m})_b})$ is a finitely generated $\zed$-graded $\Q[a]$-module. So, by Lemma \ref{lemma-fg-graded-module-structure}, $\fH_N^{\ve,\star,k}(\overline{(\tau_{i_1}\cdots\tau_{i_m})_b})$ has a unique standard decomposition. Now, to prove the lemma, we only need to verify that:
\begin{enumerate}[(I)]
	\item The free part of $\fH_N^{\ve,\star,k}(\overline{(\tau_{i_1}\cdots\tau_{i_m})_b})$ is isomorphic to $H_N^{\ve,k}(\overline{(\tau_{i_1}\cdots\tau_{i_m})_b}) \otimes_\Q \Q[a]\{-b\}_a$.
	\item All torsion components of $\fH_N^{\ve,\star,k}(\overline{(\tau_{i_1}\cdots\tau_{i_m})_b})$ are of the form $\Q[a]/(a)\{s\}_a$.
	\item If $\fH_N^{\ve,\star,k}(\overline{(\tau_{i_1}\cdots\tau_{i_m})_b})$ contains a torsion component $\Q[a]/(a)\{s\}_a$, then $-b \leq s \leq -1$ and $(N-1)s \leq k -2N+m$.
\end{enumerate}

These three conclusions can be easily proved by an induction on the weight of $\overline{(\tau_{i_1}\cdots\tau_{i_m})_b}$ using Lemmas \ref{lemma-MOY-decomps}, \ref{lemma-MOY-decomps-KR} and Corollary \ref{cor-resolved-braid-induction}. 

If the weight of a closed resolved braid is $0$, then it is $\overline{(\emptyset)_b}$. By Lemma \ref{lemma-homology-empty-braid}, (I-III) is true for $\overline{(\emptyset)_b}$ for all $b\geq 0$.

Now assume that (I-III) is true for all closed resolved braids (on any number of strands) with weight less than the weight of $\overline{(\tau_{i_1}\cdots\tau_{i_m})_b}$. By Corollary \ref{cor-resolved-braid-induction}, via a finite sequence of isotopies of types I$_1$ and I$_2$, $\overline{(\tau_{i_1}\cdots\tau_{i_m})_b}$ is isotopic to a closed resolved braid of one of the following three types:
\begin{enumerate}[(a)]
	\item $\overline{(\tau_{j_1}\cdots\tau_{j_{m-1}}\tau_i)_b}$, where $i>j_1,\dots,j_m$;
	\item $\overline{(\tau_{j_1}\cdots\tau_{j_{m-2}}\tau_j\tau_j)_b}$;
	\item $\overline{(\tau_{j_1}\cdots\tau_{j_{m-3}}\tau_j\tau_{j-1}\tau_j)_b}$.
\end{enumerate}
Of course, isotopies of types I$_1$ and I$_2$ do not change the weight of a closed resolved braid.

In Case (a), we have
\begin{eqnarray*}
\fH_N(\overline{(\tau_{i_1}\cdots\tau_{i_m})_b}) & \cong & \fH_N(\overline{(\tau_{j_1}\cdots\tau_{j_{m-1}}\tau_i)_b}), \\
\fH_N(\overline{(\tau_{j_1}\cdots\tau_{j_{m-1}})_b}) & \cong & \fH_N(\overline{(\tau_{j_1}\cdots\tau_{j_{m-1}}\tau_i)_b})\{0,1\} \oplus \fH_N(\overline{(\tau_{j_1}\cdots\tau_{j_{m-1}})_{b-1}})\{-1,1-N\}, \\
H_N(\overline{(\tau_{j_1}\cdots\tau_{j_{m-1}})_b}) & \cong & H_N(\overline{(\tau_{j_1}\cdots\tau_{j_{m-1}}\tau_i)_b})\{1\}_x \oplus H_N(\overline{(\tau_{j_1}\cdots\tau_{j_{m-1}})_{b-1}})\{1-N\}_x,
\end{eqnarray*}
where the second and third isomorphisms follow from Lemmas \ref{lemma-MOY-decomps} and \ref{lemma-MOY-decomps-KR}. The weights of both $\overline{(\tau_{j_1}\cdots\tau_{j_{m-1}})_b}$ and $\overline{(\tau_{j_1}\cdots\tau_{j_{m-1}})_{b-1}}$ are less than that of $\overline{(\tau_{i_1}\cdots\tau_{i_m})_b}$. So (I-III) are true for $\overline{(\tau_{j_1}\cdots\tau_{j_{m-1}})_b}$ and $\overline{(\tau_{j_1}\cdots\tau_{j_{m-1}})_{b-1}}$. Moreover, by Lemma \ref{lemma-fg-graded-module-structure}, the standard decomposition of $\fH_N^{\ve,\star,k}(\overline{(\tau_{j_1}\cdots\tau_{j_{m-1}})_b})$ is unique. It then follows from the above isomorphisms that (I-III) are true for $\overline{(\tau_{i_1}\cdots\tau_{i_m})_b}$ too.

In Case (b), we have
\begin{eqnarray*}
\fH_N(\overline{(\tau_{i_1}\cdots\tau_{i_m})_b}) & \cong & \fH_N(\overline{(\tau_{j_1}\cdots\tau_{j_{m-2}}\tau_j\tau_j)_b}), \\
\fH_N(\overline{(\tau_{j_1}\cdots\tau_{j_{m-2}}\tau_j\tau_j)_b}) & \cong & \fH_N(\overline{(\tau_{j_1}\cdots\tau_{j_{m-2}}\tau_j)_b})\{0,1\} \oplus \fH_N(\overline{(\tau_{j_1}\cdots\tau_{j_{m-2}}\tau_j)_b}) \{0,-1\}, \\
H_N(\overline{(\tau_{j_1}\cdots\tau_{j_{m-2}}\tau_j\tau_j)_b}) & \cong & H_N(\overline{(\tau_{j_1}\cdots\tau_{j_{m-2}}\tau_j)_b})\{1\}_x \oplus H_N(\overline{(\tau_{j_1}\cdots\tau_{j_{m-2}}\tau_j)_b}) \{-1\}_x,
\end{eqnarray*}
where the second and third isomorphisms follow from Lemmas \ref{lemma-MOY-decomps} and \ref{lemma-MOY-decomps-KR}. The weight of $\overline{(\tau_{j_1}\cdots\tau_{j_{m-2}}\tau_j)_b}$ is less than that of $\overline{(\tau_{i_1}\cdots\tau_{i_m})_b}$. So (I-III) are true for $\overline{(\tau_{j_1}\cdots\tau_{j_{m-2}}\tau_j)_b}$. It then follows from the above isomorphisms that (I-III) are true for $\overline{(\tau_{i_1}\cdots\tau_{i_m})_b}$ too.

In Case (c), we have 
\begin{eqnarray*}
\fH_N(\overline{(\tau_{i_1}\cdots\tau_{i_m})_b}) & \cong & \fH_N(\overline{(\tau_{j_1}\cdots\tau_{j_{m-3}}\tau_j\tau_{j-1}\tau_j)_b}), \\
\fH_N(\overline{(\tau_{j_1}\cdots\tau_{j_{m-3}}\tau_j\tau_{j-1}\tau_j)_b}) \oplus \fH_N(\overline{(\tau_{j_1}\cdots\tau_{j_{m-3}}\tau_{j-1})_b}) & \cong & \fH_N(\overline{(\tau_{j_1}\cdots\tau_{j_{m-3}}\tau_{j-1}\tau_{j}\tau_{j-1})_b}) \oplus \fH_N(\overline{(\tau_{j_1}\cdots\tau_{j_{m-3}}\tau_{j})_b}), \\
H_N(\overline{(\tau_{j_1}\cdots\tau_{j_{m-3}}\tau_j\tau_{j-1}\tau_j)_b}) \oplus H_N(\overline{(\tau_{j_1}\cdots\tau_{j_{m-3}}\tau_{j-1})_b}) & \cong & H_N(\overline{(\tau_{j_1}\cdots\tau_{j_{m-3}}\tau_{j-1}\tau_{j}\tau_{j-1})_b}) \oplus H_N(\overline{(\tau_{j_1}\cdots\tau_{j_{m-3}}\tau_{j})_b}),
\end{eqnarray*}
where the second and third isomorphisms follow from Lemmas \ref{lemma-MOY-decomps} and \ref{lemma-MOY-decomps-KR}. The weights of $\overline{(\tau_{j_1}\cdots\tau_{j_{m-3}}\tau_{j-1})_b}$, $\overline{(\tau_{j_1}\cdots\tau_{j_{m-3}}\tau_{j-1}\tau_{j}\tau_{j-1})_b}$ and $\overline{(\tau_{j_1}\cdots\tau_{j_{m-3}}\tau_{j})_b}$ are less than that of $\overline{(\tau_{i_1}\cdots\tau_{i_m})_b}$. So (I-III) are true for these three resolved closed braids. It then follows from the above isomorphisms that (I-III) are true for $\overline{(\tau_{i_1}\cdots\tau_{i_m})_b}$ too.
\end{proof}

\begin{corollary}\label{corollary-MOY-homology-vanish}
$\fH^{b+1,\star,\star}(\overline{(\tau_{i_1}\cdots\tau_{i_m})_b})$ is a direct sum of components of the form $\Q[a]/(a)\{s,k\}$
\end{corollary}

\begin{proof}
From \cite{KR1}, we know that $H^{b+1,\star}(\overline{(\tau_{i_1}\cdots\tau_{i_m})_b})\cong 0$. So the corollary follows from Lemma \ref{lemma-module-homology-resolved-braid}.
\end{proof}

\subsection{Homology of a closed braid}\label{subsec-homology-braid} We are now ready to prove Theorem \ref{thm-fH-module}.

Let $B$ be a closed braid of $b$ strands. Recall that $\fH_N(B) = H(H(\fC_N(B),d_{mf}),d_\chi)$. Denote by $H^{\ve,i,j,k}(\fC_N(B),d_{mf})$ the homogeneous component of $H(\fC_N(B),d_{mf})$ of $\zed_2$-degree $\ve$, homological degree $i$, $a$-degree $j$ and $x$-degree $k$. We use the $\star$-notation as introduced in Definition \ref{def-component-notations}. Then, for every $(\ve,k) \in \zed_2\oplus \zed$, $(H^{\ve,\star,\star,k}(\fC_N(B),d_{mf}), d_\chi)$ is a bounded chain complex of finitely generated $\zed$-graded $\Q[a]$-modules. Denote by $F^{\ve,i,\star,k}$ the free part of $H^{\ve,i,\star,k}(\fC_N(B),d_{mf})$ and by $T^{\ve,i,\star,k}$ the torsion part of $H^{\ve,i,\star,k}(\fC_N(B),d_{mf})$. Note that $sl(B)=c_+-c_--b$, where $c_\pm$ is the number of $\pm$ crossings in $B$. Then, by Lemma \ref{lemma-module-homology-resolved-braid}, 
\begin{itemize}
	\item $F^{\ve,i,\star,k} \cong H^{\ve,i,k}(C_N(B),d_{mf})\otimes_\Q \Q[a]\{sl(B)\}_a$,
	\item $T^{\ve,i,\star,k}$ is a direct sum of finitely many components of the form $\Q[a]/(a)\{s\}_a$.
\end{itemize}
Under the decomposition $H^{\ve,i,\star,k}(\fC_N(B),d_{mf}) =  \left.%
\begin{array}{c}
  F^{\ve,i,\star,k} \\
  \oplus \\
  T^{\ve,i,\star,k}
\end{array}%
\right.$, then differential map $H^{\ve,i,\star,k}(\fC_N(B),d_{mf}) \xrightarrow{d_\chi^i} H^{\ve,i+1,\star,k}(\fC_N(B),d_{mf})$ takes the form
\[
\left.%
\begin{array}{c}
  F^{\ve,i,\star,k} \\
  \oplus \\
  T^{\ve,i,\star,k}
\end{array}%
\right.
\xrightarrow{\left(%
\begin{array}{cc}
 d_{\chi,FF}^i & 0 \\
 d_{\chi,FT}^i & d_{\chi,TT}^i
\end{array}%
\right)}
\left.%
\begin{array}{c}
  F^{\ve,i+1,\star,k} \\
  \oplus \\
  T^{\ve,i+1,\star,k}
\end{array}%
\right.,
\]
where $d_{\chi,FF}^i$, $d_{\chi,FT}^i$ and $d_{\chi,TT}^i$ are homogeneous homomorphisms of $\zed$-graded $\Q[a]$-modules preserving the $a$-grading. This give rise to two chain complexes $(F^{\ve,\star,\star,k}, d_{\chi,FF})$ and $(T^{\ve,\star,\star,k}, d_{\chi,TT})$. Moreover, $(H^{\ve,\star,\star,k}(\fC_N(B),d_{mf}), d_\chi)$ is isomorphic to the mapping cone of the chain map $F^{\ve,\star,\star,k}\|1\| \xrightarrow{d_{\chi,FT}} T^{\ve,\star,\star,k}$, where ``$\|\ast\|$" means shifting the homological grading up by $\ast$. Thus, we get the follow lemma.

\begin{lemma}\label{lemma-FT-exact-seq}
There is a short exact sequence
\[
0\rightarrow T^{\ve,\star,\star,k} \rightarrow \fC_N^{\ve,\star,\star,k}(B) \rightarrow F^{\ve,\star,\star,k} \rightarrow 0,
\]
which induces a long exact sequence
\[
\cdots \xrightarrow{d_{\chi,FT}^{i-1}} H^i(T^{\ve,\star,\star,k}) \rightarrow \fH_N^{\ve,i,\star,k}(B) \rightarrow H^i(F^{\ve,\star,\star,k}) \xrightarrow{d_{\chi,FT}^{i}} H^{i+1}(T^{\ve,\star,\star,k}) \rightarrow \cdots
\]
of $\zed$-graded $\Q[a]$-modules, where the arrows preserve the $\Q[a]$-grading.
\end{lemma}

\begin{proof}
This lemma follows from the standard construction of a long exact sequence from a mapping cone. 
\end{proof}

\begin{lemma}\label{lemma-free-homology}
$H^i(F^{\ve,\star,\star,k}) \cong H_N^{\ve,i,k}(B)\otimes_\Q \Q[a]\{sl(B)\}_a$ for every $(\ve,i,k) \in \zed_2 \oplus \zed^{\oplus 2}$.
\end{lemma}

\begin{proof}
Recall that $C_N(B) := \fC_N(B)/(a-1)\fC_N(B)$ and $\fC_N(B)$ is a free $\Q[a]$-module. So there is a short exact sequence 
\[
0 \rightarrow \fC_N(B) \xrightarrow{a-1} \fC_N(B) \rightarrow C_N(B) \rightarrow 0.
\]
This induces a long exact sequence 
{\tiny
\[
\cdots \rightarrow H^{\ve,i,\star,\star}(\fC_N(B),d_{mf}) \xrightarrow{a-1} H^{\ve,i,\star,\star}(\fC_N(B),d_{mf}) \rightarrow H^{\ve,i,\star,\star}(C_N(B),d_{mf}) \rightarrow H^{\ve+1,i,\star,\star}(\fC_N(B),d_{mf})\{-1,-N-1\} \xrightarrow{a-1} \cdots
\]}\vspace{-1pc}

\noindent preserving the $x$-grading. By \cite[Lemma 9.1]{Wu-triple-trans}, the multiplication by $a-1$ is an injective endomorphism of $H^{\ve,i,\star,\star}(\fC_N(B),d_{mf})$. So this long exact sequence breaks into a short exact sequence
\[
0\rightarrow (H^{\ve,\star,\star,\star}(\fC_N(B),d_{mf}),d_\chi) \xrightarrow{a-1} (H^{\ve,\star,\star,\star}(\fC_N(B),d_{mf}),d_\chi) \rightarrow (H^{\ve,\star,\star,\star}(C_N(B),d_{mf}),d_\chi) \rightarrow 0.
\]
This shows that the chain complexes $(H(C_N(B),d_{mf}),d_\chi)$ and $(H(\fC_N(B),d_{mf})/(a-1)H(\fC_N(B),d_{mf}),d_\chi)$ are isomorphic to each other, and the isomorphism preserves the $\zed_2$-, homological and $x$-gradings.  

From the decomposition $H^{\ve,i,\star,k}(\fC_N(B),d_{mf}) =  \left.%
\begin{array}{c}
  F^{\ve,i,\star,k} \\
  \oplus \\
  T^{\ve,i,\star,k}
\end{array}%
\right.$, it is clear that 
\[
(H^{\ve,\star,\star,k}(\fC_N(B),d_{mf})/(a-1)H^{\ve,\star,\star,k}(\fC_N(B),d_{mf}), d_\chi) \cong (F^{\ve,\star,\star,k}/(a-1)F^{\ve,\star,\star,k}, d_{\chi,FF}).
\]
So there is an isomorphism of chain complexes
\begin{equation}\label{eq-isomorphism-quotient-free}
(F^{\ve,\star,\star,k}/(a-1)F^{\ve,\star,\star,k}, d_{\chi,FF}) \cong (H^{\ve,\star,k}(C_N(B),d_{mf}),d_\chi).
\end{equation}
Recall that $d_{\chi,FF}$ preserves the $a$-grading and $F^{\ve,i,\star,k} \cong H^{\ve,i,k}(C_N(B),d_{mf})\otimes_\Q \Q[a]\{sl(B)\}_a$, where the shift of the $a$-grading is independent of the homological grading $i$. Now let $n_i = \dim_\Q H^{\ve,i,k}(C_N(B),d_{mf})$ and fix a basis for $H^{\ve,i,k}(C_N(B),d_{mf})$. This basis induces a $\Q[a]$-basis for $F^{\ve,i,\star,k}$ and allows us to identify $F^{\ve,i,\star,k}$ with $\Q[a]^{\oplus n_i}\{sl(B)\}_a$. Thus, $(F^{\ve,\star,\star,k}, d_{\chi,FF})$ is isomorphic to the chain complex
\[
C= \cdots \xrightarrow{D_{i-1}} \Q[a]^{\oplus n_i}\{sl(B)\}_a \xrightarrow{D_{i}} \Q[a]^{\oplus n_{i+1}}\{sl(B)\}_a \xrightarrow{D_{i+1}} \cdots,
\]
where $D_i$ is the matrix of $d^i_{\chi,FF}$ relative to the bases of $F^{\ve,i,\star,k}$ and $F^{\ve,i+1,\star,k}$. Since $d_{\chi,FF}$ preserves the $a$-grading, all entries of $D_i$ are elements of $\Q$. Consider the chain complex 
\[
\hat{C}= \cdots \xrightarrow{D_{i-1}} \Q^{\oplus n_i} \xrightarrow{D_{i}} \Q^{\oplus n_{i+1}} \xrightarrow{D_{i+1}} \cdots.
\]
One can see that $(F^{\ve,\star,\star,k}, d_{\chi,FF}) \cong C \cong \hat{C} \otimes_\Q \Q[a] \{sl(B)\}_a$ and, by isomorphism \eqref{eq-isomorphism-quotient-free}, $\hat{C} \cong C/(a-1)C \cong (H^{\ve,\star,k}(C_N(B),d_{mf}),d_\chi)$. Combining these, we get 
\[
(F^{\ve,\star,\star,k}, d_{\chi,FF}) \cong (H^{\ve,\star,k}(C_N(B),d_{mf})\otimes_\Q \Q[a] \{sl(B)\}_a,d_\chi).
\] 
This implies that $H^i(F^{\ve,\star,\star,k}) \cong H_N^{\ve,i,k}(B)\otimes_\Q \Q[a]\{sl(B)\}_a$.
\end{proof}

\begin{proof}[Proof of Theorem \ref{thm-fH-module}]
From Lemmas \ref{lemma-FT-exact-seq} and \ref{lemma-free-homology}, we get a long exact sequence 
\begin{equation}\label{eq-FT-exact-seq}
\cdots \rightarrow H^i(T^{\ve,\star,\star,k}) \rightarrow \fH_N^{\ve,i,\star,k}(B) \rightarrow H_N^{\ve,i,k}(B)\otimes_\Q \Q[a]\{sl(B)\}_a \xrightarrow{d_{\chi,FT}^{i}} H^{i+1}(T^{\ve,\star,\star,k}) \rightarrow \cdots.
\end{equation}
Denote by $F\fH_N^{\ve,i,\star,k}(B)$ the free part of the $\zed$-graded $\Q[a]$-module $\fH_N^{\ve,i,\star,k}(B)$ and by $T\fH_N^{\ve,i,\star,k}(B)$ the torsion part of $\fH_N^{\ve,i,\star,k}(B)$. Then the long exact sequence \eqref{eq-FT-exact-seq} splits into two exact sequences:
\begin{eqnarray}
\label{eq-FT-exact-seq-T} && \cdots \rightarrow H^i(T^{\ve,\star,\star,k}) \rightarrow T\fH_N^{\ve,i,\star,k}(B) \rightarrow 0, \\
\label{eq-FT-exact-seq-F} && 0 \rightarrow F\fH_N^{\ve,i,\star,k}(B) \xrightarrow{f} H_N^{\ve,i,k}(B)\otimes_\Q \Q[a]\{sl(B)\}_a \xrightarrow{d_{\chi,FT}^{i}} H^{i+1}(T^{\ve,\star,\star,k}) \rightarrow \cdots.
\end{eqnarray}

Since $T^{\ve,i,\star,k}$ is a direct sum of finitely many components of the form $\Q[a]/(a)\{s\}_a$, so is $H^i(T^{\ve,\star,\star,k})$. From the exact sequence \eqref{eq-FT-exact-seq-T}, one can see that $T\fH_N^{\ve,i,\star,k}(B)$ is a quotient module of $ H^i(T^{\ve,\star,\star,k})$. Thus, $T\fH_N^{\ve,i,\star,k}(B)$ is also a direct sum of finitely many components of the form $\Q[a]/(a)\{s\}_a$. That is, 
\begin{equation}\label{eq-fH-module-tor}
T\fH_N^{\ve,i,\star,k}(B) \cong (\bigoplus_{q=1}^{n} \Q[a]/(a)\{s_q\}_a),
\end{equation}
for some finite sequence $\{s_1,\dots,s_n\}$ of integers.

Next we prove that the $\Q$-linear map 
\[
F\fH_N^{\ve,i,\star,k}(B)/(a-1)F\fH_N^{\ve,i,\star,k}(B) \xrightarrow{f} H_N^{\ve,i,k}(B)\otimes_\Q \Q[a]\{sl(B)\}_a /(a-1)H_N^{\ve,i,k}(B)\otimes_\Q \Q[a]\{sl(B)\}_a
\]
is an isomorphism. First, note that $H^{i+1}(T^{\ve,\star,\star,k})$ is a direct sum of components of the form $\Q[a]/(a)\{s\}_a$. So any multiple of $a$ in $H_N^{\ve,i,k}(B)\otimes_\Q \Q[a]\{sl(B)\}_a$ is in $\ker d_{\chi,FT}^{i} = \im f$. For any $u \in F\fH_N^{\ve,i,\star,k}(B)$ such that $f(u)=(a-1)v$ for some $v \in H_N^{\ve,i,k}(B)\otimes_\Q \Q[a]\{sl(B)\}_a$, there exits an $u'\in F\fH_N^{\ve,i,\star,k}(B)$ satisfying $f(u')=av$. Thus, 
\[
f(-(a-1)(u-u')) = -(a-1)(f(u)-f(u')) = (a-1)v = f(u).
\] 
But $F\fH_N^{\ve,i,\star,k}(B) \xrightarrow{f} H_N^{\ve,i,k}(B)\otimes_\Q \Q[a]\{sl(B)\}_a$ is injective. So $u=-(a-1)(u-u')$. This shows that the above $\Q$-linear map is injective. Second, for every $v \in H_N^{\ve,i,k}(B)\otimes_\Q \Q[a]\{sl(B)\}_a$, there is a $u \in F\fH_N^{\ve,i,\star,k}(B)$ such that $f(u)=av$. So $v=f(u) - (a-1)v$. This shows that the above $\Q$-linear map is surjective. Thus, it is an isomorphism. 

The above $\Q$-linear isomorphism implies that the rank of the $\zed$-graded free $\Q[a]$-module $F\fH_N^{\ve,i,\star,k}(B)$ is equal to $\dim_\Q H_N^{\ve,i,k}(B)$. Hence, by Lemma \ref{lemma-fg-graded-module-structure},
\begin{equation}\label{eq-fH-module-free-weak}
F\fH_N^{\ve,i,\star,k}(B) \cong \bigoplus_{p=1}^{\dim_\Q H_N^{\ve,i,k}(B)} \Q[a]\{t_p\}_a.
\end{equation}
From \cite{KR1}, we know that $H_N^{sl(B)-1,i,k}(B) \cong 0$ for any $i,k$. So, for any $i,k$,
\begin{equation}\label{eq-fH-module-free-s-1}
F\fH_N^{sl(B)-1,i,\star,k}(B) \cong 0.
\end{equation}
From the construction of $\fH_N(B)$, one can see that, when $\ve=sl(B)$, the parity of $t_p$ in \eqref{eq-fH-module-free-weak} must be the same as that of $sl(B)$. Since $F\fH_N^{sl(B),i,\star,k}(B) \xrightarrow{f} H_N^{sl(B),i,k}(B)\otimes_\Q \Q[a]\{sl(B)\}_a$ is injective and preserves the $a$-grading, we know that $t_p \geq sl(B)$ if $\ve=sl(B)$. Assume that $F\fH_N^{sl(B),i,\star,k}(B)$ contains a component $\Q[a]\{t_p\}_a$ such that $t_p\geq sl(B)+4$. Denote by $1_p$ the $1$ in $\Q[a]\{t_p\}_a$. Then $f(1_p)=a^2 v$ for some $v\in H_N^{sl(B),i,k}(B)\otimes_\Q \Q[a]\{sl(B)\}_a$. Consider the exact sequence \eqref{eq-FT-exact-seq-F}. Again, since $H^{i+1}(T^{sl(B),\star,\star,k})$ is a direct sum of finitely many components of the form $\Q[a]/(a)\{s\}_a$, one can see that $av \in \ker d_{\chi,FT}^{i} = \im f$. So there exists a $u \in F\fH_N^{sl(B),i,\star,k}(B)$ such that $f(u)=av$. Therefore, $f(1_p)=f(au)$. But $f$ is injective. This means $1_p =au$, which is a contradiction. Thus, when $\ve=sl(B)$, we have $t_p = sl(B)$ or $sl(B)+2$ for every $p$ and 
\begin{equation}\label{eq-fH-module-free}
F\fH_N^{sl(B),i,\star,k}(B) \cong (\Q[a]\{sl(B)\}_a)^{\oplus l} \oplus (\Q[a]\{sl(B)+2\}_a)^{\oplus (\dim_\Q H_N^{sl(B),i,k}(B) - l)}
\end{equation}
for some non-negative integer $l$.

By decompositions \eqref{eq-fH-module-tor}, \eqref{eq-fH-module-free-s-1} and \eqref{eq-fH-module-free}, one can see that $\fH_N^{\ve,i,\star,k}(B)$ admits a decomposition of the form given in Theorem \ref{thm-fH-module}. The uniqueness of this decomposition follows from Lemma \ref{lemma-fg-graded-module-structure}. The only things left to prove are the bounds for $s_q$. In the remainder of this proof, we show that the bound for $s_q$ in Theorem \ref{thm-fH-module} follow from the corresponding bounds in Lemma \ref{lemma-module-homology-resolved-braid}.

If we choose a resolution as in Figure \ref{crossing-res-fig} for each crossing of $B$, we get a closed resolved braid. We call such a closed resolved braid a resolution of $B$ and denote by $\mathcal{R}(B)$ the set of all resolutions of $B$. As suggested in Figure \ref{crossing-res-fig}, we call the resolution $C_\pm \leadsto \Gamma_0$ a $0$-resolution and $C_\pm \leadsto \Gamma_1$ a $\pm 1$-resolution. For $\overline{\mu} \in \mathcal{R}(B)$, assume it contains $m_{\overline{\mu},+}+m_{\overline{\mu},-}$ $2$-colored edges, where $m_{\overline{\mu},\pm}$ is the number of $2$-colored edges in $\overline{\mu}$ coming from $\pm 1$-resolutions. From the construction of $\fC_N(B)$, especially local chain complexes \eqref{eq-def-chain-crossing+} and \eqref{eq-def-chain-crossing-}, one can see that
\begin{equation}\label{eq-chain-mf-decomp}
\fC_N(B) = \bigoplus_{\overline{\mu} \in \mathcal{R}(B)} \fC_N(\overline{\mu}) \left\langle w \right\rangle \{w, (N-1)w + m_{\overline{\mu},+}-m_{\overline{\mu},-}\} \|m_{\overline{\mu},-}-m_{\overline{\mu},+}\|,
\end{equation}
where $w= c_+-c_-$ is the writhe of $B$ and ``$\|\ast\|$" means shifting the homological grading by $\ast$. From Lemma \ref{lemma-module-homology-resolved-braid}, we know that, if $\fH_N^{\ve,\star,k-(N-1)w -m_{\overline{\mu},+}+m_{\overline{\mu},-}}(\overline{\mu})\{w\}_a$ contains a torsion component $\Q[a]/(a)\{s\}_a$, then $w-b \leq s \leq w-1$ and $(N-1)s \leq k - 2N +2m_{\overline{\mu},-}$. Note that $w-b=sl(B)$ and $m_{\overline{\mu},-} \leq c_-$. So, by decomposition \eqref{eq-chain-mf-decomp}, we have that, if $T^{\ve,\star,\star,k}$ contains a component $\Q[a]/(a)\{s\}_a$, then 
\begin{equation}\label{eq-grading-bounds}
sl(B) \leq s \leq w-1 \text{ and } (N-1)s \leq k - 2N +2c_-.
\end{equation}
Therefore, if $H^i(T^{\ve,\star,\star,k})$ contains a component $\Q[a]/(a)\{s\}_a$, then $s$ satisfies the two bounds in \eqref{eq-grading-bounds}. Finally, by the exact sequence \eqref{eq-FT-exact-seq-T}, $T\fH_N^{\ve,i,\star,k}(B)$ is a quotient module of $ H^i(T^{\ve,\star,\star,k})$. So, if $\fH_N^{\ve,i,\star,k}(B)$ contains a component $\Q[a]/(a)\{s\}_a$, then $s$ satisfies the two bounds in \eqref{eq-grading-bounds}. This completes the proof of Theorem \ref{thm-fH-module}.
\end{proof}

\section{Stabilization}\label{sec-stabilization}

In this section, we study how $\fH_N$ changes under stabilization. The goal is to prove Theorem \ref{thm-stabilization}.

\subsection{Mapping cones} We now review some basic properties of mapping cones.

\begin{definition}\label{def-mapping-cone}
Let $A$, $B$ be two chain complexes of $\zed$-graded $\Q[a]$-modules and $f:A \rightarrow B$ a chain map preserving both the homological grading and the $a$-grading. Then the mapping cone $cone(f)$ is defined to be the chain complex given by:
\begin{itemize}
	\item $cone^i(f) = \left.%
\begin{array}{c}
  A^i \\
  \oplus \\
  B^{i-1}
\end{array}%
\right.$,
\item the differential $cone^i(f) \xrightarrow{d} cone^{i+1}(f)$ is the map
$\left.%
\begin{array}{c}
  A^i \\
  \oplus \\
  B^{i-1}
\end{array}%
\right.
\xrightarrow{\left(%
\begin{array}{cc}
  d_A & 0 \\
  f & d_B
\end{array}%
\right)}
\left.%
\begin{array}{c}
  A^{i+1} \\
  \oplus \\
  B^{i}
\end{array}%
\right.$, where $d_A$ and $d_B$ are the differential maps of $A$ and $B$.
\end{itemize}
\end{definition}

\begin{lemma}\label{lemma-cone-1}
Suppose that $0 \rightarrow A \xrightarrow{f} B \xrightarrow{g} C \rightarrow 0$ is a short exact sequence of chain complexes of $\zed$-graded $\Q[a]$-modules, where $f$ and $g$ preserve both the homological grading and the $a$-grading. Then, as $\zed$-graded $\Q[a]$-modules, $H^i(cone(f)) \cong H^{i-1}(C)$ and $H^i(cone(g)) \cong H^i(A)$.
\end{lemma}

\begin{proof}
Denote by $\id_A$ the identity map from $A$ to itself. Define $\alpha:cone(\id_A)\rightarrow cone(f)$ by $\left.%
\begin{array}{c}
  A^i \\
  \oplus \\
  A^{i-1}
\end{array}%
\right.
\xrightarrow{\left(%
\begin{array}{cc}
  \id_A & 0 \\
  0 & f
\end{array}%
\right)}
\left.%
\begin{array}{c}
  A^{i} \\
  \oplus \\
  B^{i-1}
\end{array}%
\right.$
and $\beta:cone(f) \rightarrow C\|1\|$ by 
$\left.%
\begin{array}{c}
  A^{i} \\
  \oplus \\
  B^{i-1}
\end{array}%
\right.
\xrightarrow{(0,g)}
C^{i-1}$.
Then $\alpha$, $\beta$ are chain maps and 
\[
0 \rightarrow cone(\id_A) \xrightarrow{\alpha} cone(f) \xrightarrow{\beta} C\|1\| \rightarrow 0
\]
is a short exact sequence. It induces a long exact sequence
\[
\cdots \rightarrow H^i(cone(\id_A)) \rightarrow H^i(cone(f)) \rightarrow H^{i-1}(C) \rightarrow H^{i+1}(cone(\id_A)) \rightarrow \cdots
\]
Since $H(cone(\id_A)) \cong 0$. This long exact sequence implies that $H^i(cone(f)) \cong H^{i-1}(C)$.

Now define $\phi:A\rightarrow cone(g)$ by 
$A^i \xrightarrow{\left(%
\begin{array}{c}
  f \\
  0
\end{array}%
\right)}
\left.%
\begin{array}{c}
  B^{i} \\
  \oplus \\
  C^{i-1}
\end{array}%
\right.$
and $\psi: cone(g) \rightarrow cone(\id_C)$ by 
$\left.%
\begin{array}{c}
  B^i \\
  \oplus \\
  C^{i-1}
\end{array}%
\right.
\xrightarrow{\left(%
\begin{array}{cc}
  g & 0 \\
  0 & \id_C
\end{array}%
\right)}
\left.%
\begin{array}{c}
  C^{i} \\
  \oplus \\
  C^{i-1}
\end{array}%
\right.$.
Then $\phi$, $\psi$ are chain maps and
\[
0 \rightarrow A \xrightarrow{\phi} cone(g) \xrightarrow{\psi} cone(\id_C) \rightarrow 0
\]
is a short exact sequence. It induces a long exact sequence 
\[
\cdots \rightarrow H^{i-1}(cone(\id_C)) \rightarrow H^i(A) \rightarrow H^i(cone(g)) \rightarrow H^{i}(cone(\id_C)) \rightarrow \cdots
\]
Since $H(cone(\id_C)) \cong 0$, this long exact sequence implies that $H^i(cone(g)) \cong H^i(A)$.
\end{proof}

\begin{lemma}\label{lemma-cone-2}
Suppose that $0 \rightarrow A \xrightarrow{f} B \xrightarrow{g} C \xrightarrow{h} D \rightarrow 0$ is an exact sequence of chain complexes of $\zed$-graded $\Q[a]$-modules, where $f$, $g$ and $h$ preserve both the homological grading and the $a$-grading. Then there is a long exact sequence of $\zed$-graded $\Q[a]$-modules
	\[
	\cdots \rightarrow H^i(A) \rightarrow H^i(cone(g)) \rightarrow H^{i-1}(D) \rightarrow H^{i+1}(A) \rightarrow \cdots
	\]
\end{lemma}

\begin{proof}
Denote by $\pi:B \rightarrow B/f(A)$ the standard quotient map. Define $\alpha: cone(\pi) \rightarrow cone(g)$ by $\left.%
\begin{array}{c}
  B^i \\
  \oplus \\
  B^{i-1}/f(A^{i-1})
\end{array}%
\right.
\xrightarrow{\left(%
\begin{array}{cc}
 \id_B & 0 \\
  0 & g
\end{array}%
\right)}
\left.%
\begin{array}{c}
  B^{i} \\
  \oplus \\
  C^{i-1}
\end{array}%
\right.$, which is well defined since $\ker g =\im f$. Also, define $\beta:cone(g) \rightarrow D\|1\|$ by $\left.%
\begin{array}{c}
  B^i \\
  \oplus \\
  C^{i-1}
\end{array}%
\right.
\xrightarrow{(0,h)}
D^{i-1}$. Then $\alpha$, $\beta$ are chain maps and 
\[
0 \rightarrow cone(\pi) \xrightarrow{\alpha} cone(g) \xrightarrow{\beta} D\|1\| \rightarrow 0
\]
is a short exact sequence. It induces a long exact sequence
\[
\cdots \rightarrow H^i(cone(\pi)) \rightarrow H^i(cone(g)) \rightarrow H^{i-1}(D) \rightarrow H^{i+1}(cone(\pi)) \rightarrow \cdots
\]
But $0 \rightarrow A \xrightarrow{f} B \xrightarrow{\pi} B/f(A) \rightarrow 0$ is a short exact sequence of complexes. So, by Lemma \ref{lemma-cone-1}, we know that $H^i(cone(\pi)) \cong H^i (A)$. Thus, we have a long exact sequence
	\[
	\cdots \rightarrow H^i(A) \rightarrow H^i(cone(g)) \rightarrow H^{i-1}(D) \rightarrow H^{i+1}(A) \rightarrow \cdots
	\]
\end{proof}

\subsection{Stabilization and $\fH_N$} Next, we prove Theorem \ref{thm-stabilization}.

\begin{proof}[Proof of Theorem \ref{thm-stabilization}]
Let $B$ be a closed braid. Set $\mathscr{C}_N(B) = \fC_N(B)/a\fC_N(B)$. Recall that $\pi_0$ is the standard quotient map $\fC_N(B) \xrightarrow{\pi_0} \fC_N(B)/a\fC_N(B) = \mathscr{C}_N(B)$. Then there is a short exact sequence
\[
0 \rightarrow \fC_N(B) \xrightarrow{a} \fC_N(B)\{-2,0\} \xrightarrow{\pi_0} \mathscr{C}_N(B) \{-2,0\} \rightarrow 0.
\]
Note that $d_{mf}$ is homogeneous with $\zed_2$-degree $1$, homological degree $0$, $a$-degree $1$ and $x$-degree $N+1$. Set $s=sl(B)$. Taking the homology with respect to $d_{mf}$, the above short exact sequence gives the following long exact sequence.
{\small
\[
\xymatrix{
&& \cdots\ar[lld]  \\
H^{s-1,\star,\star,k-N-1}(\fC_N(B), d_{mf})\{1\}_a \ar[r]^a & H^{s-1,\star,\star,k-N-1}(\fC_N(B), d_{mf})\{-1\}_a \ar[r]^{\pi_0} & H^{s-1,\star,\star,k-N-1}(\mathscr{C}_N(B), d_{mf})\{-1\}_a \ar[lld] \\
H^{s,\star,\star,k}(\fC_N(B), d_{mf})\ar[r]^a & H^{s,\star,\star,k}(\fC_N(B), d_{mf})\{-2\}_a \ar[r]^{\pi_0} & H^{s,\star,\star,k}(\mathscr{C}_N(B), d_{mf})\{-2\}_a \ar[lld] \\
H^{s-1,\star,\star,k+N+1}(\fC_N(B), d_{mf})\{-1\}_a \ar[r]^a & H^{s-1,\star,\star,k+N+1}(\fC_N(B), d_{mf})\{-3\}_a \ar[r] & \cdots
}
\]
}

Following the notations in Subsection \ref{subsec-homology-braid}, we denote by $F^{\ve,i,\star,k}$ the free part of $H^{\ve,i,\star,k}(\fC_N(B),d_{mf})$ and by $T^{\ve,i,\star,k}$ the torsion part of $H^{\ve,i,\star,k}(\fC_N(B),d_{mf})$. By Corollary \ref{corollary-MOY-homology-vanish} and the normalization of the local chain complexes \eqref{eq-def-chain-crossing+} and \eqref{eq-def-chain-crossing-}, we know that $F^{s-1,i,\star,k} \cong 0$ and $T^{\ve,i,\star,k}$ is a direct sum of components of the form $\Q[a]/(a)\{\ast\}_a$. So the above long exact sequence breaks into two exact sequences:
\begin{equation}\label{eq-exact-seq-s-1}
0\rightarrow H^{s-1,\star,\star,k-N-1}(\fC_N(B), d_{mf})\{-1\}_a \xrightarrow{\pi_0} H^{s-1,\star,\star,k-N-1}(\mathscr{C}_N(B), d_{mf})\{-1\}_a \rightarrow T^{\ve,\star,\star,k} \rightarrow 0
\end{equation}
and 
{\small
\begin{equation}\label{eq-exact-seq-s}
0\rightarrow F^{s,\star,\star,k} \rightarrow H^{s,\star,\star,k}(\fC_N(B), d_{mf})\{-2\}_a \xrightarrow{\pi_0} H^{s,\star,\star,k}(\mathscr{C}_N(B), d_{mf})\{-2\}_a \rightarrow H^{s-1,\star,\star,k+N+1}(\fC_N(B), d_{mf})\{-1\}_a \rightarrow 0.
\end{equation}
}

Applying Lemma \ref{lemma-cone-1} to the exact sequence \eqref{eq-exact-seq-s-1}, we get that 
\[
H^{s-1,i,\star,k}(cone(H(\fC_N(B),d_{mf})\xrightarrow{\pi_0}H(\mathscr{C}_N(B),d_{mf})),d_\chi)\{-1\}_a \cong H^{i-1}(T^{s,\star,\star,k+N+1},d_\chi).
\]
By \cite[Theorem 1.5]{Wu-triple-trans}, 
\[
\fH_N^{s-1,i,\star,k}(B_-) \cong H^{s-1,i,\star,k}(cone(H(\fC_N(B),d_{mf})\xrightarrow{\pi_0}H(\mathscr{C}_N(B),d_{mf})),d_\chi)\{-2\}_a.
\]
So
\[
\fH_N^{s-1,i,\star,k}(B_-) \cong H^{i-1}(T^{s,\star,\star,k+N+1},d_\chi)\{-1\}_a.
\]
By Lemmas \ref{lemma-FT-exact-seq} and \ref{lemma-free-homology}, there is a
long exact sequence
\[
\cdots \rightarrow H^i(T^{\ve,\star,\star,k}) \rightarrow \fH_N^{\ve,i,\star,k}(B) \rightarrow H_N^{\ve,i,k}(B)\otimes_\Q \Q[a]\{s\}_a \rightarrow H^{i+1}(T^{\ve,\star,\star,k}) \rightarrow \cdots
\]
Thus, we have a long exact sequence
\[
\cdots \rightarrow \fH_N^{s-1,i,\star,k}(B_-) \rightarrow \fH_N^{s,i-1,\star,k+N+1}(B)\{-1\}_a \rightarrow H_N^{s,i-1,k+N+1}(B)\otimes_\Q \Q[a]\{s-1\}_a \rightarrow \fH_N^{s-1,i+1,\star,k}(B_-) \rightarrow \cdots
\]
This establishes the long exact sequence \eqref{eq-exact-seq-stabilization-s-1}.

Now apply Lemma \ref{lemma-cone-2} to the exact sequence \eqref{eq-exact-seq-s}. Using also the fact that 
\[
\fH_N^{s,i,\star,k}(B_-) \cong H^{s,i,\star,k}(cone(H(\fC_N(B),d_{mf})\xrightarrow{\pi_0}H(\mathscr{C}_N(B),d_{mf})),d_\chi)\{-2\}_a,
\]
we get a long exact sequence
\[
\cdots \rightarrow H^i(F^{s,\star,\star,k}) \rightarrow \fH_N^{s,i,\star,k}(B_-) \rightarrow \fH_N^{s-1,i-1,\star,k+N+1}(B)\{-1\}_a \rightarrow H^{i+1}(F^{s,\star,\star,k}) \rightarrow \cdots
\]
By Lemma \ref{lemma-free-homology}, $H^i(F^{s,\star,\star,k}) \cong H_N^{s,i,k}(B)\otimes_\Q \Q[a]\{s\}_a$, which is a free $\Q[a]$-module. From \cite{KR1}, we know that $H_N^{s-1,\star,\star}(B) \cong 0$. So, by Theorem \ref{thm-fH-module}, $\fH_N^{s-1,i-1,\star,k+N+1}(B)$ is a torsion $\Q[a]$-module. Thus, the above long exact sequence breaks into the following short exact sequence.
\[
0 \rightarrow H_N^{s,i,k}(B)\otimes_\Q \Q[a]\{s\}_a \rightarrow \fH_N^{s,i,\star,k}(B_-) \rightarrow \fH_N^{s-1,i-1,\star,k+N+1}(B)\{-1\}_a \rightarrow 0.
\]
This establishes the short exact sequence \eqref{eq-exact-seq-stabilization-s}.
\end{proof}

\subsection{Transverse unknots} We are now ready to prove Corollary \ref{cor-unknots}. We start by a simple algebraic observation. 

\begin{lemma}\label{lemma-quotient}
Let $\mathcal{F} = \bigoplus_{l=0}^{N-1}\Q[a]\left\langle 1\right\rangle \{-1, -N+1+2l\}$ be as defined in Lemma \ref{cor-unknots}. 
\begin{enumerate}
	\item Assume $f:\mathcal{F} \rightarrow \mathcal{F}$ is an injective homogeneous homomorphism of $a$-degree $2$ and preserving other gradings. Then $\coker f \cong \mathcal{F}/a\mathcal{F}$.
	\item Assume $g:\mathcal{F} \rightarrow \mathcal{F}$ is an injective homogeneous homomorphism preserving all gradings. Then $g$ is an isomorphism.
\end{enumerate}
\end{lemma}

\begin{proof}
The proofs for the two parts are very similar. We only include here the proof for Part (1) and leave Part (2) for the reader. 

Denote by $1_l$ the ``$1$" in $\Q[a]\left\langle 1\right\rangle \{-1, -N+1+2l\}$. Then, since $f$ is an injective homogeneous homomorphism of $a$-degree $2$ and preserves the $x$-grading, we know that $f(1_l) = \lambda_l a 1_l$ for some $\lambda_l \in \Q \setminus \{0\}$. The lemma follows from this.
\end{proof}

\begin{proof}[Proof of Corollary \ref{cor-unknots}]
Setting $b=1$ in Lemma \ref{lemma-homology-empty-braid}, we get that $\fH_N(U_0) \cong \mathcal{F} \oplus \mathcal{T}$. 

For $m=1$, the exact sequences in Theorem \ref{thm-stabilization} are non-vanishing at only two locations:
\begin{equation}
\label{eq-exact-U-1-0}  0 \rightarrow \fH_N^{0,1,\star,\star}(U_1) \rightarrow \fH_N^{1,0,\star,\star}(U_0)\{-1, -N-1\} \rightarrow H_N^{1,0,\star}(U_0)\otimes_\Q \Q[a]\{-2,-N-1\} \rightarrow \fH_N^{0,2,\star,k}(U_1) \rightarrow 0,
\end{equation}
\begin{equation}\label{eq-exact-U-1-1} 
0 \rightarrow H_N^{1,0,\star}(U_0)\otimes_\Q \Q[a]\{-1\}_a \rightarrow \fH_N^{1,0,\star,\star}(U_1) \rightarrow 0. 
\end{equation}
Recall that, from \cite{KR1}, we know that $H_N(U_m)\cong H_N(U_0)\cong \bigoplus_{l=0}^{N-1}\Q\left\langle 1\right\rangle \{-N+1+2l\}_x$. So 
\begin{equation}\label{eq-unknot-KR}
H_N(U_m)\otimes_\Q \Q[a] \cong H_N(U_0)\otimes_\Q \Q[a] \cong \bigoplus_{l=0}^{N-1}\Q[a]\left\langle 1\right\rangle \{0, -N+1+2l\} \cong \mathcal{F} \{1\}_a. 
\end{equation}
Also, by Remark \ref{remark-torsion-s-1}, $ \fH_N^{0,1,\star,\star}(U_1)$ is a torsion $\Q[a]$-module. So exact sequence \eqref{eq-exact-U-1-0} breaks into 
\begin{eqnarray}
\label{eq-exact-U-1-2} & 0 \rightarrow \fH_N^{0,1,\star,\star}(U_1) \rightarrow \mathcal{T}\{-1, -N-1\} \rightarrow 0, & \\
\label{eq-exact-U-1-3} & 0 \rightarrow \mathcal{F}\{-1, -N-1\} \rightarrow \mathcal{F}\{-1,-N-1\} \rightarrow \fH_N^{0,2,\star,k}(U_1) \rightarrow 0. &
\end{eqnarray}
Thus, we have $\fH_N^{0,1,\star,\star}(U_1) \cong \mathcal{T}\{-1, -N-1\}$ and, by Part (2) of Lemma \ref{lemma-quotient}, $\fH_N^{0,2,\star,k}(U_1) \cong 0$. Also, using exact sequence \eqref{eq-exact-U-1-1}, we have $\fH_N^{1,0,\star,\star}(U_1) \cong \mathcal{F}$. Putting everything together, we have $\fH_N(U_1) \cong \mathcal{F} \oplus \mathcal{T} \left\langle 1 \right\rangle \{-1, -N-1\}\|1\|$

Next, assume the corollary is true for $U_m$ for some $m\geq 1$. We prove that the corollary is true for $U_{m+1}$. By \eqref{eq-unknot-KR}, $H_N^{\ve,i,\star}(U_m)\otimes_\Q \Q[a] \cong \begin{cases}
\mathcal{F} \{1\}_a & \text{if } \ve=1 \text{ and } i=0, \\
0 & \text{otherwise.} 
\end{cases}$ 
So the exact sequences in Theorem \ref{thm-stabilization} break into 
\begin{equation}\label{eq-exact-U-m-0}
0 \rightarrow \mathcal{F} \{-2m\}_a \rightarrow \fH_N^{1,0,\star,\star}(U_{m+1}) \rightarrow 0
\end{equation}
\begin{equation}\label{eq-exact-U-m-1}
0 \rightarrow \fH_N^{0,1,\star,\star}(U_{m+1}) \rightarrow \mathcal{F}\{-2m+1, -N-1\} \rightarrow \mathcal{F}\{-2m-1,-N-1\} \rightarrow \fH_N^{0,2,\star,\star}(U_{m+1}) \rightarrow 0,
\end{equation}
\begin{equation}\label{eq-exact-U-m-2}
0 \rightarrow \fH_N^{l+1,l+2,\star,\star}(U_{m+1}) \rightarrow \mathcal{F}/a\mathcal{F}  \{-2m+l-1, -(l+1)(N+1)\} \rightarrow 0, \text{ for } l=1,\dots,m-1,
\end{equation}
\begin{equation}\label{eq-exact-U-m-3}
 0 \rightarrow  \fH_N^{m-1,m+1,\star,\star}(U_{m+1})  \rightarrow \mathcal{T} \{-m-1, -(m+1)(N+1)\} \rightarrow 0.
\end{equation}
Exactness of \eqref{eq-exact-U-m-0} gives us 
\[
\fH_N^{1,0,\star,\star}(U_{m+1}) \cong \mathcal{F} \{-2m\}_a. 
\]
Exactness of \eqref{eq-exact-U-m-2} and \eqref{eq-exact-U-m-3} give us 
\begin{eqnarray*}
\fH_N^{l+1,l+2,\star,\star}(U_{m+1}) & \cong & \mathcal{F}/a\mathcal{F} \{-2m+l-1, -(l+1)(N+1)\}, \\
\fH_N^{m-1,m+1,\star,\star}(U_{m+1}) & \cong & \mathcal{T} \{-m-1, -(m+1)(N+1)\}.
\end{eqnarray*}
Finally, we look at exact sequence \eqref{eq-exact-U-m-1}. By Remark \ref{remark-torsion-s-1}, $\fH_N^{0,1,\star,\star}(U_{m+1})$ is a torsion $\Q[a]$-module. This implies that $\fH_N^{0,1,\star,\star}(U_{m+1}) \cong 0$ and we have a short exact sequence
\[
0 \rightarrow \mathcal{F}\{-2m+1, -N-1\} \rightarrow \mathcal{F}\{-2m-1,-N-1\} \rightarrow \fH_N^{0,2,\star,\star}(U_{m+1}) \rightarrow 0.
\]
Applying Part (1) of Lemma \ref{lemma-quotient} to the above short exact sequence, we get 
\[
\fH_N^{0,2,\star,\star}(U_{m+1}) \cong \mathcal{F}/a \mathcal{F} \{-2m-1,-N-1\}. 
\]
Now putting everything together, we have that 
\begin{eqnarray*}
\fH_N(U_{m+1}) & \cong & \mathcal{F} \{-2((m+1)-1),0\} \oplus \mathcal{T} \left\langle m+1 \right\rangle \{-(m+1), -(m+1)(N+1)\}\|m+1\| \\ 
&& \oplus \bigoplus_{l=1}^{(m+1)-1} \mathcal{F}/a\mathcal{F} \left\langle l \right\rangle \{-2(m+1)+l, -l(N+1)\}\|l+1\|.
\end{eqnarray*}
This shows that the corollary is true for $U_{m+1}$ too.
\end{proof}

\end{document}